\documentclass{amsart}
\usepackage[utf8]{inputenc}
\usepackage{mathrsfs}
\usepackage{amsfonts}
\usepackage{amssymb}
\usepackage{amsthm}
\usepackage{bbold}
\usepackage{tikz}
\usepackage{eucal}
\usetikzlibrary{matrix,arrows,decorations.pathmorphing}

\DeclareMathOperator{\A}{\mathsf A}
\DeclareMathOperator{\dgA}{\mathcal A}
\DeclareMathOperator{\T}{\mathsf T}
\DeclareMathOperator{\dgT}{\mathcal T}
\DeclareMathOperator{\F}{\mathsf F}
\DeclareMathOperator{\p}{\mathsf p\!}
\DeclareMathOperator{\dgF}{\mathcal F}
\DeclareMathOperator{\M}{\mathsf M}
\DeclareMathOperator{\Db}{\mathsf D^{\mathsf b}}
\DeclareMathOperator{\D1}{\mathsf D_\text{$1$}}
\DeclareMathOperator{\Dn}{\mathsf D_\text{$n$}}
\DeclareMathOperator{\Kb}{\mathsf K^{\mathsf b}}
\DeclareMathOperator{\K1}{\mathsf K_\text{$1$}}
\DeclareMathOperator{\Khp}{\mathsf K_\text{$1$}^{\hp}}
\DeclareMathOperator{\Kn}{\mathsf K_\text{$n$}}
\DeclareMathOperator{\Cb}{\mathsf C^{\mathsf b}}
\DeclareMathOperator{\C1}{\mathsf C_\text{$1$}}
\DeclareMathOperator{\Cn}{\mathsf C_\text{$n$}}
\DeclareMathOperator{\colim}{\mathsf{colim}}
\DeclareMathOperator{\limit}{\mathsf{lim}}
\DeclareMathOperator{\per}{\mathsf{per}}

\DeclareMathOperator{\Hom}{\mathsf{Hom}}
\DeclareMathOperator{\RHom}{\mathbf {R} \mathsf{Hom}}
\DeclareMathOperator{\End}{\mathsf{End}}
\DeclareMathOperator{\Ext}{\mathsf{Ext}}
\DeclareMathOperator{\thick}{\mathsf{thick}}
\DeclareMathOperator{\Ind}{\mathsf{Ind}}
\DeclareMathOperator{\Res}{\mathsf{Res}}
\DeclareMathOperator{\modules}{\mathsf{mod}}
\DeclareMathOperator{\Modules}{\mathsf{Mod}}
\DeclareMathOperator{\Gr}{\mathsf{Gr}}
\DeclareMathOperator{\dgmodules}{\mathsf{dg-mod}}
\DeclareMathOperator{\pretr}{\mathsf{pre-tr}}
\DeclareMathOperator{\Dif}{\mathsf{Dif}}
\DeclareMathOperator{\gr}{\mathsf{gr}}
\DeclareMathOperator{\op}{\mathsf{op}}
\DeclareMathOperator{\id}{\mathsf{id}}
\DeclareMathOperator{\proj}{\mathsf{proj}}
\DeclareMathOperator{\hp}{\mathsf{hp}}
\DeclareMathOperator{\characteristic}{\mathsf{char}}
\DeclareMathOperator{\Kernel}{\mathsf{Ker}}
\DeclareMathOperator{\Image}{\mathsf{Im}}

\makeatletter
\def\amsbb{\use@mathgroup \M@U \symAMSb}
\makeatother

\numberwithin{equation}{section}
\newtheorem{theorem}[equation]{Theorem}
\newtheorem{lemma}[equation]{Lemma}

\newtheorem{proposition}[equation]{Proposition}
\newtheorem*{theorem*}{Theorem}
\newtheorem*{lemma*}{Lemma}
\newtheorem*{corollary*}{Corollary}
\newtheorem*{proposition*}{Proposition}
\theoremstyle{remark}
\newtheorem*{remark}{Remark}
\theoremstyle{definition}
\newtheorem*{example}{Example}
\newtheorem*{conjecture}{Conjecture}

\title{The triangulated hull of periodic complexes}
\author{Torkil Stai}
\address{Institutt for matematiske fag, NTNU, 7491 Trondheim, Norway}
\email{torkil.stai@math.ntnu.no}

\begin{document}

\begin{abstract}
   In the terms of an `$n$-periodic derived category', we describe explicitly how the orbit category of the bounded derived category of an algebra with respect to powers of the shift functor embeds in its triangulated hull. We obtain a large class of algebras whose orbit categories are strictly smaller than their triangulated hulls and a realization of the phenomenon that an automorphism need not induce the identity functor on the associated orbit category.
\end{abstract}

\maketitle

\tableofcontents

\section{Introduction} \label{introduction}
We let $\mathbb k$ be a field, and by an algebra we mean a finite dimensional $\mathbb k$-algebra. Moreover, we restrict to algebras of finite global dimension. A module is a right module, and if $\Lambda$ and $\Lambda'$ are algebras then on a $\Lambda\text{--}\Lambda'$-bimodule, $\Lambda$ acts from the left while $\Lambda'$ acts from the right. $\Modules \Lambda$ is the category of $\Lambda$-modules with $\modules \Lambda$ ($\proj \Lambda$) as its subcategory of finitely generated (projective) modules. When $\Pi$ is a graded algebra, $\Gr \Pi$ ($\gr \Pi$) is the category of graded (finitely generated) $\Pi$-modules.

Let $\T$ be a category and $\F$ an automorphism of $\T$. The \textit{orbit category} $\T\!/\F$ has the objects of $\T$, morphism spaces given by
\[
\T\!/\F (X,Y) = \mathop{\bigoplus}_{i \in \amsbb Z } \T (X, \F^i Y)
\]
and composition of $f \in \T(X, \F^i Y)$ with $g \in \T(Y, \F^j Z)$ given by $\F^i g \circ f$. For a feeble example, view a ring $R$ as a category with one object. Then the orbit category $R/1_R$ is nothing but $R[t, t^{-1}]$. It is natural to ask under what conditions certain properties of $\T$ are inherited by $\T\!/\F$. For instance, when does a triangulated structure on the former induce one on the latter in such a way that the canonical functor $\T \to \T\!/\F$ becomes a triangle functor? This has proven to be a delicate question, and it seems hopeless to aim for an answer in the generality of arbitrary triangulated categories. The evident obstacle is the fact that it is not clear how to define cones in the orbit category, since in general a morphism $X \to Y$ will be represented by some
\[
X \to \bigoplus_{i=1}^k \F^{l_i} Y.
\]
Fortunately, the introduction of the \textit{cluster category} $\Db(\modules \Lambda)/\amsbb S \circ \Sigma^{-2}$ in \cite{MR2249625} about a decade ago, motivated Keller's construction of `triangulated hulls' of certain orbit categories of derived categories in \cite{MR2184464}. As the name suggests, the triangulated hull enjoys a universal property, indicating that we should reformulate the above question as `when does the orbit category coincide with its triangulated hull?' and that answering this is as much as we can hope for. The drawback is that the triangulated hull might come across as abstruse. At the very least, engaging in calculations can certainly be disheartening. We attempt to mend this, albeit in a very specific setting. To be more precise, the rather lenient hypotheses of Keller are satisfied by the automorphism $\Sigma^n$ of $\Db(\modules \Lambda)$ for  each positive integer $n$, meaning that there is an embedding of $\Db(\modules \Lambda)/\Sigma^n$ into a triangulated hull. Our aim is to understand this embedding explicitly.

To this end, we employ the category of $n$-periodic differential complexes and chain maps over $\Lambda$, together with its homotopy category $\Kn(\modules \Lambda)$ and derived category $\Dn(\modules \Lambda)$, both of whom are triangulated. Our exposition will focus on the case $n=1$, that is modules over the algebra of dual numbers $\Lambda[X]/(X^2)$. These were considered already in the monograph of Cartan and Eilenberg \cite{MR0077480} as `modules with differentiation', and were employed as a means to prove theorems in commutative algebra under the name `differential modules' in work by Avramov, Buchweitz and Iyengar \cite{MR2308849}, from which we shall adopt much of our notation. While the study of differential modules is certainly of independent interest, also our primary motivation comes from their applications. More explicitly, and returning to arbitrary $n$, the link from the orbit category $\Db (\modules \Lambda)/ \Sigma^n$ is given by \textit{compression of complexes}, i.e.\ associating to each
\[
0 \to X^0 \to X^1 \to \cdots \to X^{l-1} \to X^l \to 0
\]
the $n$-periodic
\[
\cdots \to \bigoplus_{i \equiv 1 \, (n)}X^i \to \bigoplus_{i \equiv 2 \, (n)}X^i \to \cdots \to \bigoplus_{i \equiv n \, (n)}X^i \to \bigoplus_{i \equiv 1 \, (n)}X^i \to \cdots
\]
with the obvious differentials. Our most basal result then states the following.

\begin{theorem*}[See \ref{theoremthehull}]
   Compression of complexes, considered as a functor
   \[
   \Delta\colon \Db (\modules \Lambda)/ \Sigma^n \to \Dn(\modules \Lambda),
   \]
   is precisely the embedding of the orbit category into its triangulated hull.
\end{theorem*}

It is not difficult to demonstrate, in purely homological terms, that $\Delta$ is dense whenever $\Lambda$ is iterated tilted. Thus, once Theorem \ref{theoremthehull} is established, it will be clear that $\Db(\modules \Lambda)/\Sigma^n$ is triangulated for such algebras (Proposition \ref{weakversionofkellersresult}). This is a weak version of \cite[Theorem 1]{MR2184464}, which not only includes the piecewise hereditary algebras, but also deals with a much larger class of automorphisms than the powers of $\Sigma$. However, an orbit category need not coincide with its triangulated hull, and examples of this behavior were exhibited already in the fundamental paper of Keller. As far as the author is aware, for arbitrary automorphisms $\F$ adhering to Keller's premises, all known examples of algebras $\Lambda$ with the property that $\Db(\modules \Lambda)/\F$ coincides with its triangulated hull, are piecewise hereditary. It is only natural to investigate to what extent this condition is necessary, and a modest start could be answering if the piecewise hereditary algebras are the only ones for which $\Db(\modules \Lambda)/\Sigma^n$ is triangulated.  In the context of $\tau_2$-finite algebras and their \textit{generalized cluster categories} (see \cite{MR2640929}), the analogous question was raised by Amiot and Oppermann in \cite{MR3024264}, and similar to Theorem $7.1$ therein, we obtain the following.

\begin{theorem*}[See \ref{orientedcyclesuffices}]
   If $\Lambda$ is non-triangular, then $\Db(\modules \Lambda)/\Sigma^n$ is strictly smaller than its triangulated hull.
\end{theorem*}

The cluster category $\Db(\modules \Lambda)/\amsbb S \circ \Sigma^{-2}$ is famously a $2$-Calabi-Yau triangulated category, meaning that its Serre functor and the square of its shift functor are isomorphic (as triangle functors). In other words, the composition $\amsbb S \circ \Sigma^{-2}$ becomes naturally isomorphic to the identity on the orbit category. As the cluster category is a most prominent orbit category, and since it looks plausible at first glance, one might be led to believe that this always happens, i.e.\ that an automorphism $\F$ on a category $\T$ always induces the identity on $\T\!/\F$. But surprisingly enough, on the orbit category $\Db(\modules \Lambda)/\Sigma^n$ the identity functor need not coincide with $\Sigma^n$. This peculiarity, which was in fact noted also in \cite{keller2008corrections}, can be attributed to a sign, and consequently we must avoid characteristic $2$ and even $n$ in order to find instances. Then, however, our new found description of the context as that of $n$-periodic complexes allows us to demonstrate plainly how the phenomenon materializes.

The paper is organized as follows. Section \ref{preliminaries} starts with recalling a few well established concepts regarding modules over differential graded (DG) categories, together with some minor observations of our own, all of which will be employed in the sequel. We then review the central construction of the triangulated hull, and note that the crucial step of `enhancing' the orbit category $\Db(\modules \Lambda)/\Sigma^n$ is in fact not as involved as in the general procedure. Lastly, we explain how to view DG modules over a DG algebra as graded modules over a related graded algebra. In Section \ref{differentialmodules} we are concerned with differential $\Lambda$-modules, or more generally $n$-periodic complexes, and their homotopy and derived categories. We introduce the key notion of compression of differential complexes, as well as the tensor product of a differential module with a complex along with its fundamental properties. We go on to develop a comprehensive theory of resolutions of differential modules, allowing us to identify the $n$-periodic derived category with a certain subcategory of the $n$-periodic homotopy category, and to explain how compression yields an embedding of the orbit category $\Db(\modules \Lambda)/\Sigma^n$ into $\Dn(\modules \Lambda)$. The first of our above cited main results (Theorem \ref{theoremthehull}) is proved in Section \ref{sectionthehull}, and put to use in Section \ref{applications}. After the foreknown Proposition \ref{weakversionofkellersresult} is attained, the investigation goes in the opposite direction, in the sense that we arrive at our second main result (Theorem \ref{orientedcyclesuffices}) and briefly discuss a conjecture. Lastly, we turn to the phenomenon of $\Sigma^n$ failing to induce the identity functor on $\Db(\modules \Lambda)/\Sigma^n$. We give a simple example in full detail, before providing an underlying technical reason for this behavior.

Throughout, most of the proofs and constructions will be written out only in terms of differential modules and the category $\Db(\modules \Lambda)/\Sigma$, but we stress that this is solely for cosmetic reasons and that our results extend in obvious ways to $n$-periodic complexes and the category $\Db(\modules \Lambda)/\Sigma^n$ for each positive integer $n$.

The author is greatly indebted to Steffen Oppermann, who not only suggested the topic, but also provided invaluable guidance at countless occations. Parts of this research was conducted during a stay at the University of Toronto in the winter of 2015, and the author is most grateful to Ragnar-Olaf Buchweitz for warm hospitality and helpful discussions.

\section{Preliminaries} \label{preliminaries}
\subsection{DG modules over DG categories} For a more comprehensive introduction and further details, see Keller's papers \cite{MR1258406, MR2275593}. Recall that a \textit{(right) DG module} over a DG category $\dgA$ is a DG functor $\dgA^{\op} \to \Dif(\mathbb k)$, where $\Dif(\mathbb k)$ denotes the DG category of chain complexes over $\mathbb k$, and that the DG modules over $\dgA$ form a DG category again, denoted by $\dgmodules \dgA$. Associated to the latter is the additive category $\mathcal{CA} = Z^0(\dgmodules \dgA)$ which carries an exact structure where the sequence
\[
0 \to L \xrightarrow m M \xrightarrow p N \to 0
\]
is a conflation if there is some $s \in \dgmodules \dgA(M,N)^0$ such that $ps = 1_N$ (equivalently, there is some $r \in \dgmodules \dgA(M,L)^0$ such that $rm=1_L$). This generalizes the `degreewise split' exact structure that can be imposed on any category of chain complexes and, indeed, $\mathcal{CA}$ is Frobenius with these conflations. The homotopy category $\mathcal H \dgA = H^0(\dgmodules \dgA)$ coincides with the stable category associated to $\mathcal C \dgA$, and is hence triangulated. Finally, the derived category $\mathcal D \dgA$ is obtained by formally inverting the class of quasi-isomorphisms in $\mathcal H \dgA$.

The \textit{Yoneda embedding}
\[
\iota \colon \dgA \to \dgmodules \dgA
\]
given by $X \mapsto X^{\wedge}= \dgA (-,X)$ identifies $\dgA$ with the subcategory of its module category consisting of the representables and, by virtue of being a DG functor, restricts to a functor $Z^0 \dgA \to \mathcal{CA}$. If the image of the latter is stable under extensions and shifts (i.e.\ syzygies and cosyzygies), then $H^0 \dgA$ becomes a triangulated subcategory of $\mathcal H \dgA$ and $\dgA$ is called \textit{pre-triangulated}. Naturally, the \textit{pre-triangulated hull} of $\dgA$ is the closure of this image under extensions, shifts and summands, denoted
\[
\dgA^{\pretr} = \thick_{\mathcal C \dgA} (\iota \dgA).
\]
The pre-triangulated hull satisfies a universal property, making it functorial and left adjoint to the inclusion of pre-triangulated DG categories in DG categories. Our first observation is that the pre-triangulated hull of a DG category completely determines its module category.

\begin{lemma} \label{comparepretriangulatedhulls}
   If $\dgA$ and $\dgA'$ are DG categories such that $\dgA^{\pretr} \cong \dgA'^{\pretr}$, then $\dgmodules \dgA \cong \dgmodules \dgA'$.
\end{lemma}
\begin{proof}
It is sufficient, as well as more elegant, to check that the categories of left DG modules are equivalent. Denoting by $(\dgA, \dgA')$ the DG functors $\dgA \to \dgA'$, the left adjointness of the pre-triangulated hull reads $(\dgA, \mathcal P) \cong (\dgA^{\pretr}, \mathcal P)$ for each pre-triangulated $\mathcal P$. In particular, since $\Dif (\mathbb k)$ is pre-triangulated,
 \begin{align*}
    \dgmodules \dgA^{\op} & = (\dgA, \Dif (\mathbb k)) \\
    & \cong (\dgA^{\pretr}, \Dif(\mathbb k)) \\
    & \cong (\dgA'^{\pretr}, \Dif(\mathbb k)) \\
    & \cong (\dgA', \Dif (\mathbb k)) \\
    & = \dgmodules \dgA'^{\op}. \qedhere
\end{align*}
 \end{proof}

Associated to a DG functor $\dgF\colon \dgA \to \dgA'$ is the obvious DG restriction functor $\dgmodules \dgA' \to \dgmodules \dgA$. What is more, $\dgF$ gives rise to a DG induction functor $\dgmodules \dgA \to \dgmodules \dgA'$ (see Drinfeld \cite[Section 14]{MR2028075} or the proof of Lemma \ref{fullyfaithfulinducesfullyfaithful} for the definition), which is left adjoint to restriction. Being DG functors, restriction and induction give functors $\Res_{\dgF}\colon\mathcal {CA}' \to \mathcal {CA}$ and $\Ind_{\dgF}\colon\mathcal {CA} \to \mathcal {CA}'$, respectively.

\begin{lemma} \label{fullyfaithfulinducesfullyfaithful}
   If a DG functor $\dgF\colon \dgA \to \dgA'$ is fully faithful, then so is the induced functor $\Ind_{\dgF}\colon \mathcal{CA} \to \mathcal{CA'}$.
\end{lemma}
\begin{proof}
   Since $(\Ind_F, \Res_F)$ is an adjoint pair, a basic property of adjunctions reduces the proof to showing that the two compose to the identity on $\mathcal{CA}$. So let $Y \in \dgA'$, $M \in \dgmodules \dgA$ and identify $\dgA$ with its essential image in $\dgA'$ under $\dgF$. By definition,
   \[
   \Ind_{\dgF} (M)(Y) = M \otimes_{\dgA} (\Res_{\dgF}(\dgA'(Y,-))) = M \otimes_{\dgA} \dgA' (Y,-)
   \]
   where we only allow objects from $\dgA$ in the last argument. More explicitly,
   \[
   \Ind_{\dgF}(M)(Y) = \bigoplus_{X \in \dgA} M(X) \mathop{\otimes}_{\dgA(X,X)} \dgA'(Y,X) \Big/ \!\sim
   \]
   where $\sim$ is what one might expect: To each $f \in \dgA(X, X')$ we associate the two induced maps $f^{\ast}\colon M(X') \to M(X)$ and $f_{\ast} \colon \dgA'(Y,X) \to \dgA'(Y,X')$. The relation is then given by $f^{\ast}(u) \otimes v \sim u \otimes f_{\ast}(v)$ for each $u \in M(X')$ and $v \in \dgA'(Y,X)$. We claim that if $Y \in \dgA$, then $\Ind_{\dgF} (M)(Y) \cong M(Y)$, which is sufficient. Denoting by $u \otimes v$ the simple tensors, the crucial observation is that
   \[
   u \otimes v = u \otimes v_{\ast}(1_Y) = v^{\ast}(u) \otimes 1_Y
   \]
   for each $u \in M(X)$ and $v \in \dgA(Y,X)$. Therefore we can identify
   \[
   \Ind_{\dgF}(M)(Y) \cong \{ v^{\ast}(u) \in M(Y) \mid u \in M(X), X \in \dgA, v \in \dgA(Y,X) \} = M(Y)
   \]
   where the last equality is clear (choose $X=Y$ and $v=1_Y$).
\end{proof}

\begin{lemma} \label{lemmainductionrestrictstopretriangulatedhulls}
    Induction restricts to a functor $\Ind_{\dgF} \colon \dgA^{\pretr} \to \dgA'^{\pretr}$.
\end{lemma}
\begin{proof}
      Suppose the DG $\dgA$-module $M$ appears in a conflation
      \[
      0 \to X^{\wedge} \to M \to Y^{\wedge} \to 0
      \]
      in $\mathcal {CA}$, with $X,Y \in \dgA$. One easily checks that induction preserves representables, in the sense that $\Ind_{\dgF} (X^{\wedge}) \cong (\dgF X)^{\wedge}$. Hence the induced sequence in $\mathcal{CA'}$ is
      \[
      0 \to (\dgF X)^{\wedge} \to \Ind_{\dgF}(M) \to (\dgF Y)^{\wedge} \to 0,
      \]
      which must also be a conflation. Indeed, degreewise split exactness is preserved by any additive graded functor of degree zero. Further, for each $U \in \dgA$, the shift of $U^{\wedge}$ appears in a conflation
      \[
      0 \to U^{\wedge} \to V^{\wedge} \to \Sigma (U^{\wedge}) \to 0
      \]
      for some $V \in \dgA$. By observations above, this induces the conflation
      \[
      0 \to (\dgF U)^{\wedge} \to (\dgF V)^{\wedge} \to \Ind_{\dgF}(\Sigma (U^{\wedge})) \to 0
      \]
      in $\mathcal{CA'}$. Hence $\Ind_{\dgF}(\Sigma (U^{\wedge}))$, appearing as an admissible cokernel of $(\dgF U)^{\wedge}$ into a projective-injective, is isomorphic to $\Sigma((\dgF U)^{\wedge})$. Repeating these arguments shows that (summands of) iterated extensions and shifts of representables in $\mathcal{CA}$ are sent by the induction functor to $\dgA'^{\pretr}$.
\end{proof}

\begin{remark}
      The induction $\Ind_{\dgF} \colon \dgA^{\pretr} \to \dgA'^{\pretr}$ is nothing but $\dgF^{\pretr}$, i.e.\ the functor `pre-triangulated hull' applied to $\dgF$. Indeed, restricting to pre-triangulated hulls, the fact that induction preserves representables means precisely that
         \begin{center}
            \begin{tikzpicture}
               \matrix(a)[matrix of math nodes,
               row sep=2.5em, column sep=3em,
               text height=1.5ex, text depth=0.25ex]
               { \dgA & \dgA^{\pretr}  \\
               \dgA' & \dgA'^{\pretr} \\ };
               \path[->,font=\scriptsize](a-1-1) edge node[above]{$\iota$} (a-1-2);
               \path[->,font=\scriptsize](a-1-1) edge node[right]{$\dgF$} (a-2-1);
               \path[->,font=\scriptsize](a-2-1) edge node[above]{$\iota$} (a-2-2);
               \path[->,font=\scriptsize](a-1-2) edge node[right]{$\Ind_{\dgF}$} (a-2-2);
            \end{tikzpicture}
         \end{center}
      is commutative. Since the functoriality of the pre-triangulated hull comes from its universal property, the claim follows.
\end{remark}

\subsection{The triangulated hull} \label{orbitcategoriesandthetriangulatedhull} For each algebra $\Lambda$ and any positive integer $n$, the automorphism $\Sigma^n \cong - \otimes_{\Lambda} \Sigma^n \Lambda$ of $\Db(\modules \Lambda)$ clearly meets the requirements of \cite{MR2184464} that ensure the existence of a \textit{triangulated hull} of the orbit category $\Db(\modules \Lambda)/\Sigma^n$, that is to say, the data of a triangulated category $\M_{\Sigma^n}$ and an embedding
\[
\Db(\modules \Lambda)/\Sigma^n \to \M_{\Sigma^n}
\]
with a certain universal property. To construct $\M_{\Sigma^n}$, recall e.g.\ from \cite{MR1055981} that a \textit{DG enhancement} of a (triangulated) category $\T$ is a (pre-triangulated) DG category $\dgT$ with a (triangle) equivalence $\T \cong H^0 \dgT$. The upshot of lifting to DG categories is that canonical choices then become available, for instance, cones are famously functorial in `enhanced' triangulated categories. What is more, a DG enhanchement of the orbit category will unveil a canonical triangulated hull. Since $\Lambda$ is of finite global dimension, the DG category $\per \Lambda$ of perfect complexes over $\Lambda$ enhances its bounded derived category, and $\Sigma^n$ clearly lifts to a DG equivalence, also denoted by $\Sigma^n$, on $\per \Lambda$. This allows us to form the orbit category $\mathcal B_{\Sigma^n}= \per \Lambda / \Sigma^n$, which is naturally a DG category and gives the desired enhancement, namely
\[
\Db(\modules)/\Sigma^n \cong H^0 \mathcal B_{\Sigma^n}.
\]
The alluded to canonical choice of triangulated hull is thus
\[
\M_{\Sigma^n} = \per \mathcal B_{\Sigma^n} \subset \mathcal D \mathcal B_{\Sigma^n},
\]
i.e.\ the triangulated subcategory of the derived category of $\mathcal B_{\Sigma^n}$ generated by the representable functors, and the embedding of the orbit category is simply Yoneda
\[
H^0 \mathcal B_{\Sigma^n} \to \mathcal D \mathcal B_{\Sigma^n}
\]
which obviously factors through $\M_{\Sigma^n}$.

\begin{remark}
The above is a special case of a broader definition by Keller. In general, any automorphism $\F$ which is a \textit{standard equivalence}, i.e.\ isomorphic to
\[
 - \otimes_{\Lambda}^{\mathbf L} Z \colon \Db(\modules \Lambda) \to \Db(\modules \Lambda)
\]
for some complex $Z$ of $\Lambda\text{--}\Lambda$-bimodules, will lift to a DG functor $\dgF \colon \per \Lambda \to \per \Lambda$ in an obvious way. However, the latter might not be an equivalence, and so in order to enhance $\Db(\modules \Lambda)/\F$ we must invoke the \textit{DG orbit category}. That is, the DG category $\mathcal B'_{\dgF}$ with the objects of $\per \Lambda$ and morphism spaces of the form
\[
\mathcal B'_{\dgF} (X,Y) = \colim_p \mathop{\bigoplus}_{k \geq 0} \per \Lambda (\dgF^k X, \dgF^p Y).
\]
Its merit is of course that, when $\F$ satisfies certain mild technical hypotheses, there is an equivalence of categories $\Db(\modules \Lambda)/\F \cong H^0 \mathcal B'_{\dgF}$. We are thus in the same situation as above, and it is clear how to obtain the triangulated hull $\M_{\F}$. It is also straightforward to verify that the above $\mathcal B_{\Sigma^n}$ is a special case of the DG orbit category. Indeed, whenever $\F$ lifts to a DG equivalence $\dgF$ on $\per \Lambda$, the orbit category $\mathcal B_{\dgF} = \per \Lambda/\dgF$ exists and is naturally a DG category again. Moreover,
\[
\mathop{\bigoplus}_{k\geq 0} \per \Lambda(\dgF^k X, \dgF^p Y) \cong \mathop{\bigoplus}_{k\geq -p} \per \Lambda (\dgF^k X, Y)
\]
in this case, and hence the directed system defining $\mathcal B'_{\dgF}(X,Y)$ is the sequence
\[
\mathop{\bigoplus}_{k \geq 0} \per \Lambda (\dgF^k X, Y) \to \mathop{\bigoplus}_{k \geq -1} \per \Lambda(\dgF^k X, Y) \to \mathop{\bigoplus}_{k \geq -2} \per \Lambda(\dgF^k X, Y) \to \cdots
\]
in which each morphism is the canonical split mono. It follows that the colimit $\mathcal B'_{\dgF}(X,Y)$ is nothing but the coproduct $\mathcal B_{\dgF}(X,Y)$.
\end{remark}

\begin{remark}
   One can construct the triangulated hull also when $\Lambda$ has infinite global dimension. The enhancement of $\Db(\modules \Lambda)$ is then the DG version of $\mathsf K^{-,\mathsf b}(\proj \Lambda)$, but otherwise the construction carries over verbatim. Similarly one can handle orbit categories of $\mathsf D(\Modules \Lambda)$ using the DG enhancement given by the homotopically projective complexes (see \cite{MR1649844}).
\end{remark}

\subsection{DG modules as graded modules}\label{dgmodulesasgradedmodules}

Although the following construction appears to complicate matters, it will prove useful in Section \ref{sectionthehull}. For a DG algebra $\dgA$ let $\dgA\langle \varepsilon \rangle$ be the graded algebra with $|\varepsilon| = 1$, and $\dgA(\varepsilon)$ the graded algebra obtained as the quotient of $\dgA\langle \varepsilon \rangle$ modulo the relations $\varepsilon^2 = 0$ and $d_{\dgA}(a) + \varepsilon a - (-1)^{|a|}a \varepsilon = 0$.

\begin{lemma} \label{dgmodulesaregradedmodules}
$\Gr \dgA(\varepsilon)$ and $\mathcal C \dgA$ are equivalent categories.
\end{lemma}
\begin{proof}
   Let $M \in \mathcal C \dgA$, i.e.\ a complex
   \[
   \cdots \to M^{i-1} \xrightarrow{d_M} M^{i} \xrightarrow{d_M} M^{i+1} \to \cdots
   \]
   with an $\dgA$-action satisfying Leibniz' rule $d_M(ma) = d_M(m)a +(-1)^{|m|}md_A(a)$. Associate to $M$ the graded $\dgA (\varepsilon)$-module $(M^i)_i$ with action given by $m \cdot a = ma$ for $a \in \dgA$, and $m \cdot \varepsilon = (-1)^{|m|}d_M(m)$. Note that this action is well-defined, since putting $\alpha = d_{\dgA}(a) + \varepsilon a - (-1)^{|a|}a \varepsilon$ we get
   \[
      m \cdot \alpha = m d_{\dgA}(a) + (-1)^{|m|}d_M(m) a - (-1)^{|m|} d_M(ma) = 0.
   \]
   On the other hand, take $N \in \Gr \dgA(\varepsilon)$, that is $N =  (N^i)_i$ with an $\dgA\langle \varepsilon \rangle$-action satisfying $n \varepsilon^2 = 0$ and $n(d_{\dgA}(a) + \varepsilon a - (-1)^{|a|}a \varepsilon)=0$. Associate to $N$ the DG $\dgA$-module given by the complex
   \[
   \cdots \to N^{i-1} \xrightarrow{(-1)^{i-1}\varepsilon} N^{i} \xrightarrow{(-1)^{i}\varepsilon} N^{i+1} \to \cdots
   \]
   with $\dgA$-action given by $n \cdot a = na$. This action satisfies Leibniz' rule, since
   \[
   d_N(n \cdot a) = (-1)^{|na|}na\varepsilon
   \]
   while
   \[
   d_N(n) \cdot a + (-1)^{|n|}n \cdot d_{\dgA}(a) = (-1)^{|n|} n \varepsilon a + (-1)^{|n|} n d_{\dgA}(a),
   \]
   and the difference of the latter two expressions is
   \[
   n((-1)^{|a|}a \varepsilon - \varepsilon a - d_{\dgA}(a)) = 0.
   \]
   Leaving morphisms unaltered, the above two assignments clearly give mutually inverse equivalences between the categories in question.
\end{proof}

\section{Differential modules}\label{differentialmodules}
This section introduces the context in which we will describe the orbit category $\Db(\modules \Lambda)/\Sigma^n$ and its triangulated hull for an algebra $\Lambda$. Our exposition focuses on $n=1$, but we point out the general versions of important notions and results along the way. At these points the reader might find it instructive to fill in details.

\subsection{Triangulated structure} \label{triangulatedstructure}
To an additive category $\A$ one can associate $\C1(\A)$, i.e.\ the category of $1$-periodic complexes and chain maps in $\A$. In other words, the objects of $\C1(\A)$ are pairs $(M,\varepsilon_M)$ with \textit{underlying object} $M \in \A$ and \textit{differential} $\varepsilon_M \in \A(M,M)$ squaring to zero, and its morphisms are those between the underlying objects that commute with the differentials involved. In \cite{TheUngradedDerivedCategory} it is shown that if $\A$ is Frobenius exact, then so is $\C1(\A)$. This means that if we call the sequence
\[
0 \to (L, \varepsilon_L) \to (M, \varepsilon_M) \to (N, \varepsilon_N) \to 0
\]
in $\C1(\A)$ a conflation whenever it admits a splitting in $\A$, then $\C1(\A)$ becomes Frobenius. The injective envelope of $(M, \varepsilon_M)$ is the middle term in the conflation
\[
\bigl(M, \varepsilon_M \bigr) \xrightarrow{\bigl(\begin{smallmatrix} \varepsilon_M \\ 1_M \end{smallmatrix}\bigr)} \Biggl(M \oplus M, \biggl(\begin{matrix} 0 & 0 \\ 1_M & 0 \end{matrix} \biggr) \Biggr) \xrightarrow{\bigl(\begin{smallmatrix} 1_M & - \varepsilon_M \end{smallmatrix}\bigr)} \bigl(M, -\varepsilon_M\bigr),
\]
and each projective-injective object is of this form. Hence the stable category $\K1(\A)$ is triangulated with suspension $\Sigma$ acting by $(M,\varepsilon_M) \mapsto (M, -\varepsilon_M)$ on objects and trivially on morphisms. The \textit{mapping cone} of $f\colon (M, \varepsilon_M) \to (N, \varepsilon_N)$ is the object
\[
C_f = \Biggl(M \oplus N, \biggl(\begin{matrix} -\varepsilon_M & 0 \\ f & \varepsilon_N \end{matrix} \biggr) \Biggr),
\]
and it is straightforward to verify that the standard triangle associated to $f$ is
\[
M \xrightarrow f N \to C_f \to \Sigma M
\]
with the canonical morphisms, as well as the fact that each conflation embeds in a triangle. Equivalently, if we say that the above $f$ is \textit{null-homotopic} whenever there is some $s \in \A(M,N)$ such that $f = s \varepsilon_M + \varepsilon_N s$, then $\K1(\A)$ is the \textit{$1$-periodic homotopy category}. When $\A$ is abelian, the \textit{homology} of $(M, \varepsilon_M)$ is the quotient $\Kernel (\varepsilon_M) / \Image (\varepsilon_M)$, giving rise to a homological functor $H\colon \K1(\A) \to \A$. The class $S$ of quasi-isomorphisms in $\K1(\A)$ is thus a multiplicative system compatible with the triangulation, which means that the \textit{$1$-periodic derived category}, i.e.\ the localization
\[
\D1(\A) = S^{-1} \K1(\A),
\]
carries a triangulated structure such that the localization functor $\K1(\A) \to \D1(\A)$ is a triangle functor. Further, e.g.\ by \cite{MR0210125}, $\D1(\A)$ admits a calculus of roofs.

Our primary engagement is with the case $\A = \modules \Lambda$. Evidently, denoting by
\[
\Lambda[\varepsilon] = \Lambda[X]/(X^2)
\]
the algebra of dual numbers, $\C1(\modules \Lambda)$ is precisely $\modules \Lambda[\varepsilon]$. We hence refer to the objects of $\C1(\modules \Lambda)$ as \textit{differential modules} and its morphisms as being $\Lambda[\varepsilon]$-linear.

\begin{remark}
For each positive integer $n$, the category $\Cn(\A)$ of $n$-periodic complexes and chain maps in $\A$ is Frobenius with respect to the `degreewise split' exact structure. Hence the $n$-periodic homotopy category $\Kn(\A)$, obtained as the stabilization of $\Cn(\A)$, and its localization $\Dn(\A)$ at quasi-isomorphisms are both triangulated. Moreover, $\Cn(\modules \Lambda)$ is nothing but the finitely generated modules over $\Lambda \otimes_{\mathbb k} I_n$, where $I_n$ is the selfinjective Nakayama algebra with $n$ vertices and Loewy length $2$.
\end{remark}

\subsection{Compression of complexes} \label{compressionofcomplexes}
The algebra $\Lambda[\varepsilon]$ is graded with $|\varepsilon|=1$, and a graded module is nothing but a complex over $\Lambda$. Hence there is a forgetful functor
\[
\Delta \colon \Cb(\modules \Lambda) \to \C1(\modules \Lambda).
\]
Explicitly, for a bounded complex
\[
X = 0 \to X^0 \xrightarrow{\partial^0} X^1 \to \cdots \to X^{l-1} \xrightarrow{\partial^{l-1}} X^l \to 0
\]
we have
\[
\Delta X = (\Delta X, \varepsilon_{\Delta X}) = \Biggl(\bigoplus_{i = 0}^l X^i, \bigoplus_{i = 0}^{l-1} \partial^i \Biggr).
\]
In \cite{MR2308849}, the authors coin the descriptive term \textit{compression} for this functor. Clearly, with respect to the obvious decomposition, the differential $\varepsilon_{\Delta X}$ is the matrix
\[
\begin{pmatrix}
0 & 0 & 0 & \cdots & 0 & 0 \\
\partial^0 & 0 & 0 & \cdots &0 & 0 \\
0 & \partial^1 & 0 & \cdots & 0 & 0 \\
\vdots & \vdots & \vdots & \ddots & \vdots &  \vdots \\
0 & 0 & 0 & \cdots & 0 & 0 \\
0 & 0 & 0 & \cdots & \partial^{l-1} & 0 \\
\end{pmatrix}.
\]
Objects in the essential image of $\Delta$ will be referred to as \textit{gradable}. Compression commutes with suspensions and cones, and moreover preserves quasi-isomorphisms, thus the same construction gives triangle functors
\[
\Delta\colon \Kb(\modules \Lambda) \to \K1(\modules \Lambda)
\]
and
\[
\Delta\colon \Db(\modules \Lambda) \to \D1(\modules \Lambda).
\]

\subsection{Tensor products} \label{tensorproducts}
For each algebra $\Lambda'$ and each bounded complex
\[
X = 0 \to X^0 \xrightarrow{\partial^0} X^1 \to \cdots \to X^{l-1} \xrightarrow{\partial^{l-1}} X^l \to 0
\]
of $\Lambda\text{--}\Lambda'$-bimodules, the \textit{tensor product} of a differential $\Lambda$-module $(M, \varepsilon_M)$ with $X$ is the differential $\Lambda'$-module
\[
M \boxtimes_{\Lambda} X = \Biggl( \bigoplus_{i=0}^l (M \otimes_{\Lambda}X^i), m \otimes x \mapsto m \otimes \partial(x) + (-1)^{|x|} \varepsilon_M(m) \otimes x \Biggr).
\]
This construction, together with the obvious action on morphisms, gives a functor
\[
 - \boxtimes_{\Lambda} X \colon \C1(\modules \Lambda) \to \C1(\modules \Lambda').
 \]
The tensor product commutes with compression, in the sense that when $Y$ is a complex of right $\Lambda$-modules, then there is a canonical isomorphism
\[
(\Delta Y) \boxtimes_{\Lambda} X \cong \Delta (Y \otimes_{\Lambda} X)
\]
of differential $\Lambda'$-modules. Further, if $\Lambda''$ is also an algebra and $Z$ is a complex of $\Lambda'\text{--}\Lambda''$-bimodules, then there is a canonical isomorphism
\begin{equation} \label{tensorproductisassociative}
M \boxtimes_{\Lambda} (X \otimes_{\Lambda'} Z) \cong (M \boxtimes_{\Lambda} X) \boxtimes_{\Lambda'} Z
\end{equation}
of differential $\Lambda''$-modules. What is more, when $X$ satisfies the expected hypotheses (i.e.\ when each $X^i$ is flat as $\Lambda$-module), then $- \boxtimes_{\Lambda} X$ preserves quasi-isomorphisms and hence defines a functor $\D1(\modules \Lambda) \to \D1(\modules \Lambda')$ making the diagram
\begin{equation} \label{tensorproductcommuteswithcompressiondiagram}
   \begin{tikzpicture}[baseline=(current bounding box.center)]
  \matrix(a)[matrix of math nodes,
  row sep=2.5em, column sep=3em,
  text height=1.5ex, text depth=0.25ex]
  {\Db(\modules \Lambda) & \Db(\modules \Lambda') \\
  \D1(\modules \Lambda) & \D1(\modules \Lambda') \\};
  \path[->,font=\scriptsize](a-1-1) edge node[above]{$- \otimes_{\Lambda} X$} (a-1-2)
  (a-2-1) edge node[above]{$- \boxtimes_{\Lambda} X$}(a-2-2)
  (a-1-1) edge node[right]{$\Delta$}(a-2-1)
  (a-1-2) edge node[right]{$\Delta$}(a-2-2);
  \end{tikzpicture}
\end{equation}
commutative. The following fundamental property will be exploited in Section \ref{applications}.

\begin{lemma} \label{derivedequivalenceinducesperiodicequivalence}
   If $X$ gives rise to an equivalence
   \[
   - \otimes_{\Lambda} X\colon \Db(\modules \Lambda) \to \Db(\modules \Lambda'),
   \]
   then also
   \[
   - \boxtimes_{\Lambda} X \colon \D1(\modules \Lambda) \to \D1(\modules \Lambda').
   \]
   is invertible.
\end{lemma}
\begin{proof}
   A quasi-inverse of $- \otimes_{\Lambda} X$ is
   \[
   \RHom_{\Lambda'}(X,-) \cong - \otimes_{\Lambda'} Y,
   \]
   where $Y$ is a $\Lambda'$-projective resolution of $\RHom_{\Lambda'}(X, \Lambda')$. We claim that $- \boxtimes_{\Lambda'} Y$ is a quasi-inverse of $- \boxtimes_{\Lambda} X$, and will demonstrate how the two compose to the identity on $\D1(\modules \Lambda)$. To this end, observe that the stalk complex $\Lambda$ serves as a one-sided tensor unit, in the sense that for each differential module $M$ there is an isomorphism
   \[
   M \boxtimes_{\Lambda} \Lambda \cong M,
   \]
   natural in $M$. Further, quasi-isomorphic (bounded) complexes clearly give rise to naturally isomorphic tensor product functors. Combining these observations with the isomorphisms
   \[
   X \otimes_{\Lambda'} Y \cong \RHom_{\Lambda'}(X,X) \cong \Lambda
   \]
   in $\Db(\modules \Lambda)$, it follows from (\ref{tensorproductisassociative}) that there are natural isomorphisms
   \[
   (M \boxtimes_{\Lambda} X) \boxtimes_{\Lambda'} Y \cong M \boxtimes_{\Lambda} (X \otimes_{\Lambda'} Y) \cong M \boxtimes_{\Lambda} \Lambda = M
   \]
   in $\D1(\modules \Lambda)$. A similar argument shows that the reversed composition is isomorphic to the identity on $\D1(\modules \Lambda')$.
\end{proof}

\subsection{Resolutions} \label{resolutions} Using the terminology of \cite{MR2308849}, a \textit{projective flag} in a differential module $(P, \varepsilon_P)$ is a decomposition of the underlying module $P = P_l \oplus P_{l-1} \oplus \dots \oplus P_0$ where each $P_i \in \proj \Lambda$, with respect to which $\varepsilon_P$ is of the form
\begin{equation} \label{strictlylowertriangularmatrix}
   \begin{pmatrix}
   0 & 0 & 0 & \cdots & 0 & 0 \\
   \partial^{l,l-1} & 0 & 0 & \cdots &0 & 0 \\
   \partial^{l,l-2} & \partial^{l-1,l-2} & 0 & \cdots & 0 & 0 \\
   \vdots & \vdots & \vdots & \ddots & \vdots &  \vdots \\
   \partial^{l,1} & \partial^{l-1,1} & \partial^{l-2,1} & \cdots & 0 & 0 \\
   \partial^{l,0} & \partial^{l-1,0} & \partial^{l-2,0} & \cdots & \partial^{1,0} & 0 \\
   \end{pmatrix}.
\end{equation}
Not only are the differential modules admitting projective flags homologically most viable (e.g.\ Lemma \ref{localizationrestrictstofullyfaithful} and \cite[Proposition 2.4]{MR2308849}), they also provide resolutions in the following sense.

\begin{lemma} \label{htpresolutionsexist}
   Each differential module $(M,\varepsilon_M) \in \C1(\modules \Lambda)$ is quasi-isomorphic to one admitting a projective flag.
\end{lemma}
\begin{proof}
   Choose finite projective resolutions
   \[
   0 \to X_{l'} \to \cdots \to X_1 \to X_0 \to \Image (\varepsilon_M) \to 0
   \]
   and
   \[
   0 \to Y_{l''} \to \cdots \to Y_1 \to Y_0 \to H(M) \to 0
   \]
   in $\modules \Lambda$. Applying the Horseshoe Lemma, first to
   \[
   0 \to \Image (\varepsilon_M) \to \Kernel (\varepsilon_M) \to H(M) \to 0
   \]
   and then to
   \[
   0 \to \Kernel (\varepsilon_M) \to M \to \Image (\varepsilon_M) \to 0,
   \]
   produces a projective resolution
   \[
   0 \to P_l \xrightarrow{\partial_l} P_{l-1} \to \cdots \to P_1 \xrightarrow{\partial_1} P_0 \xrightarrow{\partial_0} M \to 0
   \]
   of $M$ in $\modules \Lambda$, in which $P_i = X_i \oplus Y_i \oplus X_i$. By construction, equipping each $P_i$ with the endomorphism
   \[
   \hat{\varepsilon}_i = \begin{pmatrix}
   0 & 0 & 0 \\
   0 & 0 & 0 \\
   1_{X_i} & 0 & 0 \\
      \end{pmatrix}
   \]
   turns the latter resolution into a sequence of $\Lambda[\varepsilon]$-linear morphisms. Introducing the sign $\varepsilon_i = (-1)^i \hat{\varepsilon}_i$ results in the diagram
   \begin{equation}\label{widediagram}
   \begin{tikzpicture}[baseline=(current bounding box.center)]
      \matrix(a)[matrix of math nodes,
      row sep=2em, column sep=2em,
      text height=1.5ex, text depth=0.25ex]
      {0 & P_l & P_{l-1} &\cdots & P_1 & P_0 & M & 0 \\
      0 & P_l & P_{l-1} &\cdots & P_1 & P_0 & M & 0 \\};
      \path[->,font=\scriptsize](a-1-1) edge (a-1-2)
      (a-1-2) edge node[above]{$\partial_l$}(a-1-3)
      (a-1-3) edge (a-1-4)
      (a-1-4) edge (a-1-5)
      (a-1-5) edge node[above]{$\partial_1$}(a-1-6)
      (a-1-6) edge node[above]{$\partial_0$}(a-1-7)
      (a-1-7) edge (a-1-8)

      (a-2-1) edge (a-2-2)
      (a-2-2) edge node[above]{$\partial_l$}(a-2-3)
      (a-2-3) edge (a-2-4)
      (a-2-4) edge (a-2-5)
      (a-2-5) edge node[above]{$\partial_1$}(a-2-6)
      (a-2-6) edge node[above]{$\partial_0$}(a-2-7)
      (a-2-7) edge (a-2-8)

      (a-1-2) edge node[right]{$\varepsilon_l$}(a-2-2)
      (a-1-3) edge node[right]{$\varepsilon_{l-1}$}(a-2-3)
      (a-1-5) edge node[right]{$\varepsilon_1$}(a-2-5)
      (a-1-6) edge node[right]{$\varepsilon_0$}(a-2-6)
      (a-1-7) edge node[right]{$\varepsilon_M$}(a-2-7);
   \end{tikzpicture}
   \end{equation}
   in which the rows are exact and each square is anti-commutative, except for the rightmost one, which is commutative. With respect to the indicated decomposition of the $\Lambda$-module $\p M = P_l \oplus P_{l-1} \oplus \cdots \oplus P_0$, define
   \[
   \varepsilon_{\p M} = \begin{pmatrix}
   \varepsilon_l & 0 & 0 & \cdots & 0 & 0 \\
   \partial_l & \varepsilon_{l-1} & 0 & \cdots &0 & 0 \\
   0 & \partial_{l-1} & \varepsilon_{l-2} & \cdots & 0 & 0 \\
   \vdots & \vdots & \vdots & \ddots & \vdots &  \vdots \\
   0 & 0 & 0 & \cdots & \varepsilon_1 & 0 \\
   0 & 0 & 0 & \cdots & \partial_1 & \varepsilon_0 \\
   \end{pmatrix}.
   \]
   Now $\varepsilon_{\p M}$ squares to zero because of the change of sign, and the assignment
   \[
   (p_l, p_{l-1} ,\dots, p_0) \mapsto \partial_0(p_0)
   \]
   defines a $\Lambda[\varepsilon]$-linear $f\colon (\p M, \varepsilon_{\p M}) \to (M, \varepsilon_M)$ by commutativity of the rightmost square in (\ref{widediagram}). Keeping this diagram in mind, it is straightforward to verify that $f$ is a quasi-isomorphism, which suffices since $\varepsilon_{\p M}$ is of the required form.
   \end{proof}

We say that a differential $\Lambda$-module $K$ is \textit{homotopically projective} if
\[
\Hom_{\K1(\Lambda)}(K,N)=0
\]
for each acyclic $N$. The class of homotopically projectives is closed under extensions and hence constitutes a triangulated subcategory
\[
\Khp(\modules \Lambda) \subset \K1(\modules \Lambda).
\]
It is easy to check that each $(Q,0)$ with $Q \in \proj \Lambda$ is homotopically projective, from which it follows that the homotopically projectives encompass the differential modules admitting projective flags. Indeed, each object in the latter class can be obtained as an iterated extension of differential modules with underlying projective module and vanishing differential. To show this, note that the differential $\varepsilon_P$ in (\ref{strictlylowertriangularmatrix}) on the underlying $P = P_l \oplus P_{l-1} \oplus \dots \oplus P_0$ restricts to a differential, also denoted by $\varepsilon_P$, on each summand of $P$ of the form $P_i \oplus P_{i-1} \oplus\cdots\oplus P_0$ for $i = 0, \dots , l$. Moreover, these differentials fit in a filtration
\begin{align} \label{filtration}
(P_0, 0) \subset (P_1 \oplus P_0, \varepsilon_P) \subset \cdots \subset (P_{l-1} \oplus \cdots \oplus P_0, \varepsilon_P) \subset (P, \varepsilon_P)
\end{align}
in $\C1(\modules \Lambda)$ with the property that each filtration factor has vanishing differential. Hence, iterating from the canonical conflation
\[
0 \to (P_0,0) \to (P_1 \oplus P_0, \varepsilon_P) \to (P_1, 0) \to 0
\]
we reach each term of (\ref{filtration}), including $(P, \varepsilon_P)$ itself, as an extension
\[
0 \to (P_{i-1} \oplus\cdots\oplus P_0, \varepsilon_P) \to (P_i \oplus\cdots\oplus P_0, \varepsilon_P) \to (P_i,0) \to 0.
\]

\begin{lemma} \label{localizationrestrictstofullyfaithful}
   If $P$ is a homotopically projective differential $\Lambda$-module, then the localization functor induces an isomorphism
   \[
   \Hom_{\K1(\Lambda)}(P, M) \cong \Hom_{\D1(\Lambda)}(P,M)
   \]
   for each $M$.
\end{lemma}
\begin{proof}
   First, observe that if $f \in \Hom_{\K1(\Lambda)}(P,M)$ maps to the zero morphism in $\D1(\modules \Lambda)$, then it factors through some acyclic $N$ and hence vanishes already in $\K1(\modules \Lambda)$, since $\Hom_{\K1(\Lambda)}(P, N)=0$. Second, in the calculus of roofs, a morphism $P \to M$ in $\D1(\modules \Lambda)$ is represented by a diagram
   \begin{center}
   \begin{tikzpicture}
         \matrix(a)[matrix of math nodes,
         row sep=1em, column sep=1em,
         text height=1.5ex, text depth=0.25ex]
         {& X & \\
         P && M \\};
         \path[->,font=\scriptsize](a-1-2) edge node[pos=0.7,above]{$q$} (a-2-1);
         \path[->,font=\scriptsize](a-1-2) edge node[pos=0.7,above]{$f$} (a-2-3);
   \end{tikzpicture}
   \end{center}
   with $q$ a quasi-isomorphism. In the triangle
   \[
   X \xrightarrow q P \to C_q \to \Sigma X,
   \]
   the middle morphism must be zero, as $C_q$ is acyclic. Thus $q$ is a split epimorphism, and there is some $\hat{q}$ ensuring commutativity of both squares in
   \begin{center}
   \begin{tikzpicture}
   \matrix(a)[matrix of math nodes,
   row sep=1em, column sep=1em,
   text height=1.5ex, text depth=0.25ex]
   {   &   & P &   &   \\
       & X &   & P &   \\
     P &   &   &   & M . \\};
   \path[font=\scriptsize]
       (a-1-3) edge [->] node[pos=0.7,above]{$\hat{q}$} (a-2-2) edge [->] node[pos=0.7,above]{$1$} (a-2-4)
       (a-2-4) edge [->] node[below]{$1$} (a-3-1)
               edge [->] node[pos=0.7,above]{$f\hat{q}$} (a-3-5)
       (a-2-2) edge [-,line width=4pt,draw=white] (a-3-5) edge [->] node[below]{$f$} (a-3-5)
               edge [->] node[pos=0.7,above]{$q$} (a-3-1);
   \end{tikzpicture}
   \end{center}
   Hence the roofs
   \begin{center}
   \begin{tikzpicture}
   \matrix(a)[matrix of math nodes,
   row sep=1em, column sep=1em,
   text height=1.5ex, text depth=0.25ex]
   {& X & \\
   P && M \\};

   \draw(1.8,0) node{and};

   \path[->,font=\scriptsize](a-1-2) edge node[pos=0.7,above]{$q$} (a-2-1);
   \path[->,font=\scriptsize](a-1-2) edge node[pos=0.7,above]{$f$} (a-2-3);
   \end{tikzpicture}
   \begin{tikzpicture}
   \matrix(a)[matrix of math nodes,
   row sep=1em, column sep=1em,
   text height=1.5ex, text depth=0.25ex]
   {& P & \\
   P && M \\};
   \path[->,font=\scriptsize](a-1-2) edge node[pos=0.7,above]{$1$} (a-2-1);
   \path[->,font=\scriptsize](a-1-2) edge node[pos=0.7,above]{$\,\, f\hat{q}$} (a-2-3);
   \end{tikzpicture}
   \end{center}
   are equivalent, and it follows that the left hand roof lies in the image of the localization functor, since the right hand one clearly does.
\end{proof}

Note that the incredibly convenient attribute of homotopically projectives expressed in Lemma \ref{localizationrestrictstofullyfaithful} is a $1$-periodic analog of the essential property enjoyed by projective resolutions of differential complexes. Lemma \ref{htpresolutionsexist} hence indicates that, up to homotopy, the class of differential modules admitting a projective flag should coincide with the class of homotopically projectives.

\begin{proposition} \label{htprojisprojflag}
   The class of homotopically projectives in $\K1(\modules \Lambda)$ is, up to isomorphism, precisely the class of differential modules admitting projective flags. Moreover, restriction of the localization functor gives a triangle equivalence
   \[
   \Khp(\modules \Lambda) \cong \D1(\modules \Lambda).
   \]
\end{proposition}

\begin{proof}
   We have already seen how the differential modules admitting projective flags are all homotopically projective. Conversely, by Lemma \ref{htpresolutionsexist} each differential module $K$ allows a quasi-isomorphism $\p K \to K$ where $\p K$ admits a projective flag, and by Lemma \ref{localizationrestrictstofullyfaithful} this will be invertible already in $\K1(\modules \Lambda)$ whenever $K$ is homotopically projective. The last claim is clear, as density of the restricted localization functor is Lemma \ref{htpresolutionsexist}, while full faithfulness is Lemma \ref{localizationrestrictstofullyfaithful}.
\end{proof}

\begin{remark}
   With no restriction on the global dimension of $\Lambda$, it is clear how to devise homotopically projective resolutions of differential $\Lambda$-modules whose underlying modules are not necessarily finitely generated. Indeed, using the idea of the proof of Lemma \ref{htpresolutionsexist}, the resolutions will appear as colimits of possibly infinite systems
   \[
   P_0 \to P_1 \to P_2 \ \to \cdots
   \]
   of differential modules in which each map is split mono over $\Lambda$ and each quotient $P_{i+1}/P_i$ has vanishing differential and underlying projective module. Plainly, as in the filtration (\ref{filtration}), each $P_i$ is homotopically projective, and so it follows from
   \[
   \Hom_{\K1(\Lambda)}(\colim P_i, N) \cong \limit \Hom_{\K1(\Lambda)}(P_i, N)
   \]
   that also the colimit itself satisfies the required vanishing condition. Hence, in analogy to the unbounded resolutions of differential complexes of \cite{MR1649844}, there is a triangle equivalence $\D1(\Modules \Lambda) \cong \Khp(\Modules \Lambda)$.
\end{remark}

Let us denote by $\K1(\proj \Lambda)$ the triangulated subcategory of $\K1(\modules \Lambda)$ consisting of \textit{relatively projectives}, i.e.\ differential modules whose underlying modules are projective. In a sense, the most convenient scenario imaginable, and also the most striking analogy to the context of differential complexes and their resolutions, is that the relatively projectives are all homotopically projective. A priori, this would be just as surprising as it would be beneficial. For instance, let $G$ be the quiver
\begin{center}
   \begin{tikzpicture}
      \matrix(a)[matrix of math nodes,
      row sep=2em, column sep=2em,
      text height=1.5ex, text depth=0ex]
      {1 & 2 & 3 \\};
      \path[->,font=\scriptsize](a-1-1) edge node[above]{$\alpha$}(a-1-2)
      (a-1-2) edge node[above]{$\beta$}(a-1-3);
      \path[->,font=\scriptsize](a-1-3) edge[in=315, out=225, looseness = .3] node[below]{$\gamma$} (a-1-1);
   \end{tikzpicture}
\end{center}
and consider the algebra $\mathbb k G/(\beta \alpha)$. Then, denoting by $P_i$ the indecomposable projective corresponding to vertex $i$, it might seem unreasonable to expect that the relatively projective differential module $(P_2, \alpha \gamma \beta)$ admits a projective flag, even in the $1$-periodic homotopy category. We nevertheless have the following.

\begin{proposition} \label{htprojisrelproj}
   The categories $\Khp(\modules \Lambda)$ and $\K1(\proj \Lambda)$ coincide.
\end{proposition}
\begin{proof}
   Each homotopically projective in $\K1(\modules \Lambda)$ is isomorphic to a differential module admitting a projective flag, and so is relatively projective. Conversely, let $(P_0, \varepsilon)$ be relatively projective. Observe that $\Image (\varepsilon)$ admits a projective resolution
   \[
   P = 0 \to P_l \xrightarrow{\partial_l} P_{l-1} \to \cdots \to P_1 \xrightarrow{\partial_1} P_0 \to 0,
   \]
   obtained by splicing a finite resolution of $\Kernel (\varepsilon)$ with its inclusion into $P_0$. Let $(\Delta P, \varepsilon_{\Delta P})$ be the compression of $P$. The cone
   \[
   C_{-1_{\Delta P}} = \Biggl(\Delta P \oplus \Delta P, \begin{pmatrix} -\varepsilon_{\Delta P} & 0 \\ -1_{\Delta P} & \varepsilon_{\Delta P} \end{pmatrix} \Biggr)
   \]
   vanishes in the homotopy category, and it therefore suffices to show that
   \[
   \bigl( Q, \varepsilon_Q \bigr) = \bigl(P_0, \varepsilon\bigr) \oplus C_{-1_{\Delta P}}
   = \Biggl(\Delta P \oplus P_0 \oplus \Delta P, \begin{pmatrix} -\varepsilon_{\Delta P} & 0 & 0 \\ 0 & \varepsilon & 0  \\ -1_{\Delta P} & 0 & \varepsilon_{\Delta P} \end{pmatrix} \Biggr)
   \]
   is isomorphic to a differential module admitting a projective flag. To this end,
   \begin{center}
   \begin{tikzpicture}[baseline=(current bounding box.center)]
         \matrix(a)[matrix of math nodes,
         row sep=2em, column sep=2em,
         text height=1.5ex, text depth=0.25ex]
         {0 & P_l & P_{l-1} &\cdots & P_1 & P_0 & \Image (\varepsilon_Q) & 0 \\
         0 & P_l & P_{l-1} &\cdots & P_1 & P_0 & \Image (\varepsilon_Q) & 0 \\};
         \path[->,font=\scriptsize](a-1-1) edge (a-1-2)
         (a-1-2) edge node[above]{$\partial_l$}(a-1-3)
         (a-1-3) edge (a-1-4)
         (a-1-4) edge (a-1-5)
         (a-1-5) edge node[above]{$\partial_1$}(a-1-6)
         (a-1-6) edge node[above]{$\varepsilon$}(a-1-7)
         (a-1-7) edge (a-1-8)

         (a-2-1) edge (a-2-2)
         (a-2-2) edge node[above]{$\partial_l$}(a-2-3)
         (a-2-3) edge (a-2-4)
         (a-2-4) edge (a-2-5)
         (a-2-5) edge node[above]{$\partial_1$}(a-2-6)
         (a-2-6) edge node[above]{$\varepsilon$}(a-2-7)
         (a-2-7) edge (a-2-8)

         (a-1-7) edge node[right]{$0$}(a-2-7)
         (a-1-6) edge node[right]{$\varepsilon$}(a-2-6)
         (a-1-5) edge node[right]{$0$}(a-2-5)
         (a-1-3) edge node[right]{$0$}(a-2-3)
         (a-1-2) edge node[right]{$0$}(a-2-2);
   \end{tikzpicture}
   \end{center}
   is a pivotal diagram. Indeed, its commutativity reveals that the indicated chain map $P \to P$ is a lift of the zero endomorphism on $\Image(\varepsilon_Q)$, and must therefore be null-homotopic. This means there are $s_i \colon P_i \to P_{i+1}$ for $0 \leq i \leq l-1$ such that $\partial_1 s_0 = \varepsilon$, $\partial_{i+1} s_i + s_{i-1} \partial_i = 0$ for $1 \leq i \leq l-1$ and $s_{l-1} \partial_l = 0$. Denoting by $e \colon \Delta P \to P_0$ and $m \colon P_0 \to \Delta P$ the canonical split epi- and monomorphism, respectively, and letting $h_i = (-1)^i s_i$, it is clear that the latter assemble to $h \colon \Delta P \to \Delta P$ such that
   \begin{equation} \label{equationhomotopywithsign}
   m \varepsilon e = \varepsilon_{\Delta P} h - h \varepsilon_{\Delta P}.
   \end{equation}
   Miraculously, it turns out that $(Q, \varepsilon_Q)$ is isomorphic to
   \[
   \bigl( Q', \varepsilon_{Q'} \bigr)
   = \Biggl(\Delta P \oplus P_0 \oplus \Delta P, \begin{pmatrix} -\varepsilon_{\Delta P} & 0 & 0 \\ - \varepsilon e & 0 & 0  \\ -(1_{\Delta P}+h) & -m & \varepsilon_{\Delta P} \end{pmatrix} \Biggr),
   \]
   whose differential is clearly of the required lower triangular form. Indeed,
   \[
   f = \begin{pmatrix} 1_{\Delta P}+h & m & 0 \\ e & 0 & -e \varepsilon_{\Delta P}  \\ 0 & 0 & 1_{\Delta P} \end{pmatrix}
   \]
   gives an isomorphism $(Q', \varepsilon_{Q'}) \to (Q, \varepsilon_Q)$. The reader is invited to check that $\varepsilon_Q f = f \varepsilon_{Q'}$, using (\ref{equationhomotopywithsign}) together with the obvious equalities
   \[
   \varepsilon_{\Delta P}m = 0, \enspace \varepsilon e \varepsilon_{\Delta P} =0 = eh \enspace \text{and} \enspace em=1_{P_0}. \qedhere
   \]
\end{proof}

\begin{example}
   Let us revisit the algebra $\mathbb k G/(\beta \alpha)$ with the relatively projective differential module $(P_2, \varepsilon = \alpha\gamma\beta)$, discussed just prior to Proposition \ref{htprojisrelproj}. The image of $\varepsilon$ is the simple module $S_2$, so the relevant diagram is
   \begin{center}
   \begin{tikzpicture}[baseline=(current bounding box.center)]
         \matrix(a)[matrix of math nodes,
         row sep=2em, column sep=2em,
         text height=1.5ex, text depth=0.25ex]
         {0 & P_1 & P_2 & S_2 & 0 \\
          0 & P_1 & P_2 & S_2 & 0. \\};
         \path[->,font=\scriptsize](a-1-1) edge (a-1-2)
         (a-1-2) edge node[above]{$\alpha$}(a-1-3)
         (a-1-3) edge node[above]{$\varepsilon$}(a-1-4)
         (a-1-4) edge (a-1-5)

         (a-2-1) edge (a-2-2)
         (a-2-2) edge node[above]{$\alpha$}(a-2-3)
         (a-2-3) edge node[above]{$\varepsilon$}(a-2-4)
         (a-2-4) edge (a-2-5)

         (a-1-4) edge node[right]{$0$}(a-2-4)
         (a-1-3) edge node[right]{$\varepsilon$}(a-2-3)
         (a-1-2) edge node[right]{$0$}(a-2-2);
   \end{tikzpicture}
   \end{center}
   In the notation of the proof of Proposition \ref{htprojisrelproj}, $(P_2, \varepsilon)$ is hence a summand, with null-homotopic complement, of the homotopically projective differential module
   \[
   \bigl( Q', \varepsilon_{Q'} \bigr) = \Biggl(P_1 \oplus P_2 \oplus P_2 \oplus P_1 \oplus P_2, \begin{pmatrix}
      0 & 0 & 0 & 0 & 0 \\ - \alpha & 0 & 0 & 0 & 0 \\ 0 & - \varepsilon & 0 & 0 & 0 \\ -1_{P_1} & -\gamma \beta & 0 & 0 & 0 \\ 0 & -1_{P_2} & -1_{P_2} & \alpha & 0 \\
   \end{pmatrix}\Biggr).
   \]
   We moreover know that $\varepsilon_{Q'}$ is obtained by conjugation of the differential of
   \[
   \bigl(Q, \varepsilon_Q\bigr) = \Biggl( P_1 \oplus P_2 \oplus P_2 \oplus P_1 \oplus P_2,
   \begin{pmatrix}
      0 & 0 & 0 & 0 & 0 \\ - \alpha & 0 & 0 & 0 & 0 \\ 0 & 0 & \varepsilon & 0 & 0 \\ -1_{P_1} & 0 & 0 & 0 & 0 \\ 0 & -1_{P_2} & 0 & \alpha & 0 \\
   \end{pmatrix}\Biggr),
   \]
   and the reader may verify this by checking that  $f^{-1} \varepsilon_Q f = \varepsilon_{Q'}$ for
   \[
   f = \begin{pmatrix}
      1_{P_1} & \gamma \beta & 0 & 0 & 0 \\ 0 & 1_{P_2} & 1_{P_2} & 0 & 0 \\ 0 & 1_{P_2} & 0 & -\alpha & 0 \\ 0 & 0 & 0 & 1_{P_1} & 0 \\ 0 & 0 & 0 & 0 & 1_{P_2} \\
   \end{pmatrix} \enspace \text{with} \enspace
   f^{-1} = \begin{pmatrix}
      1_{P_1} & 0 & - \gamma \beta & 0 & 0 \\ 0 & 0 & 1_{P_2} & \alpha & 0 \\ 0 & 1_{P_2} & -1_{P_2} & -\alpha & 0 \\ 0 & 0 & 0 & 1_{P_1} & 0 \\ 0 & 0 & 0 & 0 & 1_{P_2} \\
   \end{pmatrix}.
   \]
\end{example}

\begin{remark}
   The results of the current subsection clearly extend beyond the case $n=1$. In particular combining counterparts of Proposition \ref{htprojisprojflag} and Proposition \ref{htprojisrelproj} yields
   \[
   \Dn(\modules \Lambda) \cong \Kn(\proj \Lambda).
   \]
\end{remark}

\subsection{An embedding of the orbit category} \label{anembeddingoftheorbitcategory}
If $X,Y \in \Db(\modules \Lambda)$ then, since $\Lambda$ is of finite global dimension, we may replace $X$ by a projective resolution to get
\[
\bigoplus_{i \in \amsbb Z}\Hom_{\Db(\Lambda)}(X, \Sigma^i Y) \cong \bigoplus_{i \in \amsbb Z}\Hom_{\Kb(\Lambda)}(X, \Sigma^i Y).
\]
In this case, $\Delta X$ is homotopically projective by Proposition \ref{htprojisrelproj} and hence
\[
\Hom_{\D1(\Lambda)}(\Delta X, \Delta Y) \cong \Hom_{\K1(\Lambda)} (\Delta X, \Delta Y)
\]
by Lemma \ref{localizationrestrictstofullyfaithful}. Further, it is straightforward to write down an isomorphism
\[
\bigoplus_{i \in \amsbb Z}\Hom_{\Kb(\Lambda)}(X, \Sigma^i Y) \to \Hom_{\K1(\Lambda)} (\Delta X, \Delta Y),
\]
as an element of the left hand side is just a sequence $(f_i)_i$ of, up to homotopy, sequences of morphisms $f_i = (f_i^j)_j$ with $f_i^j\colon X^j \to Y^{j+i}$. The above three isomorphisms clearly comprise a proof of the following.

\begin{lemma} \label{lemmacompressionembedding}
   Compression of complexes induces a fully faithful embedding
   \[
   \Delta\colon \Db(\modules \Lambda)/ \Sigma \to \D1(\modules \Lambda).
   \]
   In other words, the orbit category is equivalent to the full subcategory of gradable objects in $\D1(\modules \Lambda)$. \qed
\end{lemma}

\begin{remark}
   There is a fully faithful embedding, abusively denoted by
   \[
   \Delta \colon \Db(\modules \Lambda)/\Sigma^n \to \Dn(\modules \Lambda),
   \]
   for each positive integer $n$, given by taking a complex
   \[
   0 \to X^0 \to X^1 \to \cdots \to X^{l-1} \to X^l \to 0
   \]
   to the $n$-periodic
   \[
   \cdots \to \bigoplus_{i \equiv 1 \, (n)}X^i \to \bigoplus_{i \equiv 2 \, (n)}X^i \to \cdots \to \bigoplus_{i \equiv n \, (n)}X^i \to \bigoplus_{i \equiv 1 \, (n)}X^i \to \cdots
   \]
   whose differentials, with respect to the obvious decompositions, are matrices of `gradable' shape. The embedding of Lemma \ref{lemmacompressionembedding} is clearly the case $n=1$.
\end{remark}

\section{Main result on the triangulated hull} \label{sectionthehull}
The aim of the current section is to show that the embedding $\Delta$ of Lemma \ref{lemmacompressionembedding}, and more generally the remark following it, is exactly the embedding of the orbit category into its triangulated hull (Theorem \ref{theoremthehull}). For the sake of brevity, we stick to our scheme of providing details only for the case $n=1$.

From here on, $\mathcal B$ denotes the DG category $\mathcal B_{\Sigma} = \per \Lambda/\Sigma$ from Subsection \ref{orbitcategoriesandthetriangulatedhull}, and $\mathcal B_0$ is its full DG subcategory on the single object $\Lambda$ viewed as a stalk complex. For $X, Y \in \mathcal B$, note that each degree of the `mapping complex' $\mathcal B(X,Y)$ is simply $\Hom_{\Lambda}(\Delta X, \Delta Y)$ and that the differential $\mathcal B(X,Y)^i \to \mathcal B(X,Y)^{i+1}$ is given by
\[
f \mapsto \varepsilon_{\Delta Y} f - (-1)^i f \varepsilon_{\Delta X}.
\]
Indeed,
\[
\per \Lambda (X,\Sigma^j Y)^i = \bigoplus_{l \in \amsbb Z} \Hom_{\Lambda}(X^l, Y^{l+j+i}),
\]
so letting $j$ run through the integers we obtain
\[
\mathcal B(X,Y)^i = \bigoplus_{j \in \amsbb Z} \per \Lambda(X, \Sigma^j Y)^i= \Hom_{\Lambda}(\Delta X, \Delta Y)
\]
for each $i \in \amsbb Z$. It is straightforward to check that the differential acts as claimed.

\begin{lemma} \label{equivalenceofdgcategories}
   There is an equivalence of DG categories
   \[
   \dgmodules \mathcal B \cong \dgmodules \mathcal B_0.
   \]
\end{lemma}
\begin{proof}
By Lemma \ref{lemmainductionrestrictstopretriangulatedhulls}, the inclusion $\mathcal B_0 \hookrightarrow \mathcal B$ induces a functor
\[
\nu\colon \mathcal B_0^{\pretr} \to \mathcal B^{\pretr},
\]
and if the latter is dense, then it is an equivalence by Lemma \ref{fullyfaithfulinducesfullyfaithful}, which would suffice by Lemma \ref{comparepretriangulatedhulls}. In our specific setup, i.e.\ where the objects of the DG category $\mathcal B$ are complexes, shifts and certain cones exist already in $Z^0 \mathcal B$. Indeed, the notion of shift is the obvious one, and it is clear how to define the cone of each morphism
\[
f=(f_i)_i \in Z^0 \mathcal B(X,Y) = \bigoplus_{i \in \amsbb Z} \Hom_{\mathsf{C}(\Lambda)}(X, \Sigma^i Y)
\]
with the property that $f_i$ is non-zero for at most one $i$. This is in stark contrast to the general setup where we must pass to the module category before such constructions become available. For a subcategory $\mathcal E \subset \mathcal B$, denote by $\thick_{\mathcal B}(\mathcal E)$ the full subcategory of $\mathcal B$ whose objects are summands of shifts of objects of $\mathcal E$ and cones of maps from $Z^0 \mathcal E$ of the form of the above $f$. The key observation is that
\[
\thick_{\mathcal B}(\mathcal B_0) = \mathcal B.
\]
To see why this holds, note first that each $P^i$ in a perfect complex
\[
P = 0 \to P^0 \xrightarrow{\partial^0} P^1 \to \cdots \to P^{l-1} \xrightarrow{\partial^{l-1}} P^l \to 0
\]
belongs to $\thick_{\mathcal B}(\mathcal B_0)$. Further,
\[
0 \to P^0 \xrightarrow{\partial^0} P^1 \to 0
\]
is the cone of $\partial^0 \in Z^0 \mathcal B(P^0, P^1)$, and hence also belongs to $\thick_{\mathcal B}(\mathcal B_0)$. Inductively we obtain $P$ as the cone of the chain map
\begin{center}
   \begin{tikzpicture}[baseline=(current bounding box.center)]
   \matrix(a)[matrix of math nodes,
   row sep=2em, column sep=2em,
   text height=1.5ex, text depth=0.25ex]
   {0 & P^0 & P^1 & \cdots & P^{l-2} & P^{l-1} & 0 \\
    &  &  & & 0 & P^l & 0 \\};
   \path[->,font=\scriptsize](a-1-1) edge (a-1-2)
   (a-1-2) edge node[above]{$\partial^0$}(a-1-3)
   (a-1-3) edge (a-1-4)
   (a-1-4) edge (a-1-5)
   (a-1-5) edge node[above]{$\partial^{l-2}$}(a-1-6)
   (a-1-6) edge (a-1-7)

   (a-2-5) edge (a-2-6)
   (a-2-6) edge (a-2-7)

   (a-1-6) edge node[right]{$\partial^{l-1}$}(a-2-6);
   \end{tikzpicture}
\end{center}
between objects in $\thick_{\mathcal B}(\mathcal B_0)$. Moreover, taking shifts and cones commutes with the Yoneda embedding, in the sense that there are canonical isomorphisms
\[
(\Sigma X)^{\wedge} \cong \Sigma (X^{\wedge}) \enspace \text{and} \enspace (C_f)^{\wedge} \cong C_{f^{\wedge}}
\]
in $\mathcal{CB}$ for each $X \in \mathcal B$ and $f \in Z^0 \mathcal B(X,Y)$ of the above form. Combining these observations with the fact that $\nu$ is exact and commutativity of
\begin{center}
   \begin{tikzpicture}[baseline=(current bounding box.center)]
   \matrix(a)[matrix of math nodes,
   row sep=2.5em, column sep=3em,
   text height=1.5ex, text depth=0.25ex]
   {\mathcal B_0 & \mathcal B  \\
    {\mathcal B_0^{\pretr}} & \mathcal B^{\pretr} \\};
   \path[right hook->](a-1-1) edge (a-1-2);
   \path[->,font=\scriptsize](a-1-1) edge node[right]{$\iota$}(a-2-1)
   (a-1-2) edge node[right]{$\iota$}(a-2-2)
   (a-2-1) edge node[above]{$\nu$}(a-2-2);
   \end{tikzpicture}
\end{center}
it is straightforward to verify that $\mathcal B^{\pretr} = \nu(\mathcal B_0^{\pretr})$.
\end{proof}

The point of view that a category is merely a `ring with several objects' is justified also in the differential graded context. Hence we may identify the DG category $\mathcal B_0$ with the DG algebra $\Gamma = \mathcal B(\Lambda, \Lambda)$, which in combination with Lemma \ref{equivalenceofdgcategories} reads
\[
\dgmodules \mathcal B \cong \dgmodules \Gamma.
\]
$\Gamma$ is clearly the formal DG algebra $\Lambda[t, t^{-1}]$ with $|t|=1$, and we proceed by observing that DG $\Gamma$-modules are nothing but modules over the algebra of dual numbers.

\begin{lemma} \label{equivalanceofexactcategories}
   There is an equivalence of exact categories
   \[
   \mathcal C \Gamma \cong \C1(\Modules \Lambda).
   \]
\end{lemma}
\begin{proof}
 A graded ring $\Pi$ is said to be \textit{strongly graded} if $\Pi_i \Pi_j = \Pi_{i+j}$ for all $i, j \in \amsbb Z$. By a classical theorem of Dade \cite{MR593823}, $\Pi$ is strongly graded if and only if the functor
   \[
   - \otimes_{\Pi_0} \Pi \colon \Modules \Pi_0 \to \Gr \Pi
   \]
is an equivalence. In this case, a quasi-inverse takes a graded $\Pi$-module $M$ to the $\Pi_0$-module $M_0$, and we also have $\modules \Pi_0 \cong \gr \Pi$. Recall from Lemma \ref{dgmodulesaregradedmodules} that there is an equivalence $\mathcal C \Gamma \cong \Gr \Gamma(\varepsilon)$ where, since $\Gamma$ is formal,
   \[
   \Gamma (\varepsilon) = \Gamma\langle\varepsilon\rangle/(\varepsilon^2, \varepsilon \gamma -(-1)^{|\gamma|}\gamma \varepsilon)
   \]
is the graded algebra with $|\varepsilon|=1$. It is easy to check that $\Gamma(\varepsilon)$ is strongly graded, and hence the proof is completed by the straightforward calculation
   \[
   \Gamma(\varepsilon)_0 \cong \Lambda [X]/(X^2). \qedhere
   \]
\end{proof}

Upon passage to the level of derived categories, the two previous lemmas reveal a triangle equivalence
\[
\mathcal{DB} \cong \D1(\Modules \Lambda),
\]
which sets us up for proving Theorem \ref{theoremthehull}. First, under the DG equivalence of Lemma \ref{equivalenceofdgcategories}, $\per \mathcal B$ corresponds to $\per \Gamma$. Moreover, the equivalence of Lemma \ref{equivalanceofexactcategories} identifies the free $\Gamma$-module of rank one with $(\Lambda,0)$, making the exact category $\per \mathcal B$ equivalent to the full subcategory of $\C1(\modules \Lambda)$ whose class of objects is the iterated extensions of objects of the form $(P, 0)$ with $P \in \proj \Lambda$. We have seen, using the filtration (\ref{filtration}), that this is precisely the class of differential modules admitting projective flags, and hence passing to the derived level it follows that the triangulated hull $\M_{\Sigma}$ is all of $\D1(\modules \Lambda)$. On the other hand, each perfect complex $X \in H^0 \mathcal B \cong \Db(\modules \Lambda)/\Sigma$ sits in $\dgmodules \mathcal B$ as the associated representable functor, and corresponds by Lemma \ref{equivalenceofdgcategories} to the DG $\Gamma$-module $\mathcal B(\Lambda, X)$. Passing through the equivalence of Lemma \ref{equivalanceofexactcategories} amounts simply to taking degree zero, and so $X$ is further identified with
\[
\mathcal B(\Lambda,X)^0 = \bigoplus_{i \in \amsbb Z} \per \Lambda (\Lambda, \Sigma^i X)^0 = \bigoplus_{i \in \amsbb Z} \Hom_{\Lambda}(\Lambda, X^i) \cong \Delta X,
\]
equipped with the differential $\varepsilon_{\Delta X}$.

\begin{theorem} \label{theoremthehull}
   Compression of complexes, considered as a functor
   \[
   \Delta\colon \Db (\modules \Lambda)/ \Sigma \to \D1(\modules \Lambda),
   \]
   is precisely the embedding of the orbit category into its triangulated hull. \qed
\end{theorem}
\begin{remark}
   More generally, for each positive integer $n$, the embedding
   \[
   \Delta \colon \Db(\modules \Lambda)/\Sigma^n \to \Dn(\modules \Lambda)
   \]
   described in the remark following Lemma \ref{lemmacompressionembedding}, is precisely the embedding of the orbit category into its triangulated hull. This is of course dense, and hence an equivalence, if and only if the compression $\Delta\colon \Db(\modules \Lambda) \to \Dn(\modules \Lambda)$ is dense.
\end{remark}

\section{Applications} \label{applications}

\subsection{Iterated tilted algebras}
A first application is Proposition \ref{weakversionofkellersresult} which recovers a weak version of \cite[Theorem 1]{MR2184464}. Recall that if $\mathcal H$ is a hereditary abelian category, then in $\Db(\mathcal H)$ there is an isomorphism
\[
X \cong \bigoplus_{i \in \amsbb Z} \Sigma^{-i}H^i (X)
\]
for each $X$. We start by observing that the analogous property holds in $\D1(\mathcal H)$.

\begin{lemma} \label{eachisisotostalkwhenhereditary}
   If $\mathcal H$ is hereditary abelian, then in $\D1(\mathcal H)$ there is an isomorphism
   \[
   \bigl( M, \varepsilon_M \bigr) \cong \bigl( H(M), 0 \bigr)
   \]
   for each $(M, \varepsilon_M)$.
\end{lemma}
\begin{proof}
The idea is of course to think of global dimension $i$ not as $\Ext^{i+1}$ vanishing, but rather as $\Ext^i$ being right exact. The assumption on $\mathcal H$ hence tells us that the epimorphism $M \xrightarrow{\varepsilon_M} \Image (\varepsilon_M)$ induces the exact
\[
\Ext^1_{\mathcal H}(H(M),M) \to \Ext^1_{\mathcal H}(H(M), \Image (\varepsilon_M)) \to 0.
\]
In particular there is a commutative diagram
\begin{equation} \label{surjectivemapofextensions}
   \begin{tikzpicture}[baseline=(current bounding box.center)]
   \matrix(a)[matrix of math nodes,
   row sep=2em, column sep=2em,
   text height=1.5ex, text depth=0.25ex]
   {0 & M & E & H(M) & 0 \\
   0 & \Image (\varepsilon_M) & \Kernel (\varepsilon_M) & H(M) & 0 \\};
   \path[->,font=\scriptsize](a-1-1) edge (a-1-2)
   (a-1-2) edge node[above]{$m$}(a-1-3)
   (a-1-3) edge node[above]{$p$}(a-1-4)
   (a-1-4) edge (a-1-5)

   (a-2-1) edge (a-2-2)
   (a-2-2) edge (a-2-3)
   (a-2-3) edge (a-2-4)
   (a-2-4) edge (a-2-5)

   (a-1-2) edge node[right]{$\varepsilon_M$}(a-2-2)
   (a-1-3) edge node[right]{$s$}(a-2-3)
   (a-1-4) edge node[right]{$1_{H(M)}$}(a-2-4);
   \end{tikzpicture}
\end{equation}
in $\mathcal H$ with exact rows. Notice that the chain map
\begin{center}
   \begin{tikzpicture}[baseline=(current bounding box.center)]
   \matrix(a)[matrix of math nodes,
   row sep=2em, column sep=2em,
   text height=1.5ex, text depth=0.25ex]
   {0 & M & E & 0 \\
   0 & 0 & H(M) & 0 \\};
   \path[->,font=\scriptsize](a-1-1) edge (a-1-2)
   (a-1-2) edge node[above]{$m$}(a-1-3)
   (a-1-3) edge (a-1-4)

   (a-2-1) edge (a-2-2)
   (a-2-2) edge (a-2-3)
   (a-2-3) edge (a-2-4)

   (a-1-3) edge node[right]{$p$}(a-2-3);
   \end{tikzpicture}
\end{center}
is a quasi-isomorpism, and hence
\[
\bigl( H(M),0 \bigr) \cong \Biggl( M \oplus E, \begin{pmatrix} 0 & 0 \\ m & 0 \\ \end{pmatrix}\Biggr)
\]
in $\D1(\mathcal H)$ by compression. The kernel and image of the rightmost differential is $E$ and the image of $m$, respectively, viewed as subobjects of $M \oplus E$ in the obvious ways. The map induced in homology by
\[
f  \colon \Biggl( M \oplus E, \begin{pmatrix} 0 & 0 \\ m & 0 \\ \end{pmatrix} \Biggr) \xrightarrow{( \begin{smallmatrix} 1_M & s \\ \end{smallmatrix})} \bigl( M, \varepsilon_M \bigr)
\]
therefore appears as the cokernel
\begin{center}
   \begin{tikzpicture}[baseline=(current bounding box.center)]
   \matrix(a)[matrix of math nodes,
   row sep=2em, column sep=2em,
   text height=1.5ex, text depth=0.25ex]
   {0 & \Image (m) & E & H(M \oplus E) & 0 \\
   0 & \Image (\varepsilon_M) & \Kernel (\varepsilon_M) & H(M) & 0. \\};
   \path[->,font=\scriptsize](a-1-1) edge (a-1-2)
   (a-1-2) edge (a-1-3)
   (a-1-3) edge (a-1-4)
   (a-1-4) edge (a-1-5)

   (a-2-1) edge (a-2-2)
   (a-2-2) edge (a-2-3)
   (a-2-3) edge (a-2-4)
   (a-2-4) edge (a-2-5)

   (a-1-2) edge node[right]{$s$}(a-2-2)
   (a-1-3) edge node[right]{$s$}(a-2-3)
   (a-1-4) edge node[right]{$H f$}(a-2-4);
   \end{tikzpicture}
\end{center}
This, however, is nothing but diagram (\ref{surjectivemapofextensions}), which means that $H f$ is invertible.
\end{proof}

\begin{proposition} \label{weakversionofkellersresult}
   The orbit category $\Db(\modules \Lambda)/\Sigma$ is triangulated whenever $\Lambda$ is an iterated tilted algebra.
\end{proposition}
\begin{proof}
   It suffices to show that the compression $\Delta\colon \Db(\modules \Lambda) \to \D1 (\modules \Lambda)$ is dense. This is easy, since when $\Lambda$ is iterated tilted there is a hereditary algebra $\Lambda'$ and a standard equivalence $\Db(\modules \Lambda') \to \Db(\modules \Lambda)$. From (\ref{tensorproductcommuteswithcompressiondiagram}) and Lemma \ref{derivedequivalenceinducesperiodicequivalence}, this fits in a commutative diagram
   \begin{center}
      \begin{tikzpicture}[baseline=(current bounding box.center)]
     \matrix(a)[matrix of math nodes,
     row sep=2.5em, column sep=3em,
     text height=1.5ex, text depth=0.25ex]
     {\Db(\modules \Lambda') & \Db(\modules \Lambda) \\
     \D1(\modules \Lambda') & \D1(\modules \Lambda) \\};
     \path[->,font=\scriptsize](a-1-1) edge (a-1-2)
     (a-2-1) edge (a-2-2)
     (a-1-1) edge node[right]{$\Delta$}(a-2-1)
     (a-1-2) edge node[right]{$\Delta$}(a-2-2);
     \end{tikzpicture}
  \end{center}
   where also the bottom row is an equivalence. By Lemma \ref{eachisisotostalkwhenhereditary} it is clear that the left hand compression functor is dense, forcing density of the right hand one.
\end{proof}
\begin{remark}
   In both Lemma \ref{eachisisotostalkwhenhereditary} and Proposition \ref{weakversionofkellersresult}, one may replace $\Sigma$ by $\Sigma^n$.
\end{remark}

\subsection{Non-gradable objects}
We now turn to the perhaps more intriguing problem of finding algebras $\Lambda$ for which $\Db(\modules \Lambda)/\Sigma^n$ does \textit{not} coincide with its triangulated hull. Akin to \cite[Theorem 7.1]{MR3024264}, it turns out that the existence of an oriented cycle in the ordinary quiver of $\Lambda$ is sufficient (Theorem \ref{orientedcyclesuffices}). A fundamental tool for proving this is the following.

\begin{proposition} \label{existenceofperiodiccomplexsufficient}
   If there exists an indecomposable periodic complex of finitely generated projective $\Lambda$-modules, then the embedding of $\Db(\modules \Lambda)/\Sigma$ into its triangulated hull is not dense.
\end{proposition}
\begin{proof}
   It suffices, using the resolutions of differential modules from Subsection \ref{resolutions}, to show that the compression functor
   \[
   \Delta\colon \Kb(\proj \Lambda) \to \K1(\proj \Lambda)
   \]
   is not dense, i.e.\ to produce an object in $\K1(\proj \Lambda)$ that is not gradable. So let
   \[
   Y = \cdots \to Y^l \xrightarrow{\partial^l} Y^1 \xrightarrow{\partial^1} Y^2 \to \cdots \to Y^{l-1} \xrightarrow{\partial^{l-1}} Y^l \xrightarrow{\partial^l} Y^1 \to \cdots
   \]
   be an indecomposable complex of finitely generated projectives. Then $Y$ is indecomposable also after we factor out homotopy, and the coproduct
   \[
   \bigoplus_{i=1}^l \Sigma^i Y
   \]
   is $1$-periodic, that is a differential module. Moreover, as such it cannot be gradable. Indeed, assume it belongs to the essential image of $\Delta$. This means there is some indecomposable $X \in \Kb(\proj \Lambda)$ such that
   \[
   \bigoplus_{i=1}^l \Sigma^i Y \cong \bigoplus_{i \in \amsbb Z} \Sigma^i X
   \]
   in $\K1(\proj \Lambda)$. Since $\End_{\Kb(\Lambda)}(X)$ is local, this implies that $X$ must be a summand of at least one of the left hand summands. This is a contradiction, as $Y$ cannot have a summand in $\Kb(\proj \Lambda)$.
\end{proof}

\begin{theorem} \label{orientedcyclesuffices}
   If $\Lambda$ is non-triangular, then $\Db(\modules \Lambda)/\Sigma$ is strictly smaller than its triangulated hull.
\end{theorem}
\begin{proof}
      When $\Lambda$ is non-triangular, its ordinary quiver contains an oriented cycle
      \begin{center}
         \begin{tikzpicture}
            \matrix(a)[matrix of math nodes,
            row sep=2em, column sep=2em,
            text height=1.5ex, text depth=0ex]
            {1 & 2 & 3 & \cdots & l-1 & l.\\};
            \path[->,font=\scriptsize](a-1-1) edge node[above]{$\alpha_1$}(a-1-2)
            (a-1-2) edge node[above]{$\alpha_2$}(a-1-3)
            (a-1-3) edge (a-1-4)
            (a-1-4) edge (a-1-5)
            (a-1-5) edge node[above]{$\alpha_{l-1}$}(a-1-6);
            \path[->,font=\scriptsize](a-1-6) edge [in=315, out=225, looseness = .17] node[below]{$\alpha_l$} (a-1-1);
         \end{tikzpicture}
      \end{center}
      The following algorithm produces an indecomposable periodic complex of finitely generated projective $\Lambda$-modules, which is sufficient by Proposition \ref{existenceofperiodiccomplexsufficient}. Start with the doubly infinite periodic sequence
      \[
      \cdots \to P_l \xrightarrow{\alpha_l} P_1 \xrightarrow{\alpha_1} P_2 \xrightarrow{\alpha_2} P_3 \to \cdots  \to P_{l-1} \xrightarrow{\alpha_{l-1}} P_l \xrightarrow{\alpha_l} P_1 \to \cdots.
      \]
      If this is a complex, the algorithm terminates. If not, insert the maximal non-zero composition
      \[
       a_1 \colon P_{i_1} = P_1 \xrightarrow{\alpha_1} P_2 \to \cdots \to P_{i-1} \xrightarrow{\alpha_{i-1}} P_i = P_{i_2}
      \]
      with the property $\alpha_i a_1 = 0$ into the previous sequence to obtain
      \[
      \cdots \to P_l \xrightarrow{\alpha_l} P_{i_1} \xrightarrow{a_1} P_{i_2} \xrightarrow{\alpha_i} P_{i+1} \to \cdots  \to P_{l-1} \xrightarrow{\alpha_{l-1}} P_l \xrightarrow{\alpha_l} P_1 \to \cdots.
      \]
      If this is a complex, the algorithm terminates. If not, insert the maximal non-zero composition
      \[
       a_2 \colon P_{i_2} \xrightarrow{\alpha_i} P_{i+1} \to \cdots \to P_{j-1} \xrightarrow{\alpha_{j-1}} P_j = P_{i_3}
      \]
      with the property $\alpha_j a_2 = 0$ into the previous sequence to obtain
      \[
      \cdots \to P_l \xrightarrow{\alpha_l} P_{i_1} \xrightarrow{a_1} P_{i_2} \xrightarrow{a_2} P_{i_3} \xrightarrow{\alpha_j} P_{j+1} \to \cdots \to P_{l-1} \xrightarrow{\alpha_{l-1}} P_l \xrightarrow{\alpha_l} P_1 \to \cdots.
      \]
      After finitely many steps this gives a sequence
      \[
      \cdots \to P_{i_1} \xrightarrow{a_1} P_{i_2} \to \cdots  \to P_{i_{s-1}} \xrightarrow{a_{s-1}} P_{i_s} \xrightarrow{\alpha_k} P_{k+1} \to \cdots \to P_l \xrightarrow{\alpha_l} P_{i_1} \to \cdots
      \]
      in which $\alpha_k a_{s-1} = 0$ and $a_s = \alpha_l \cdots \alpha_k \neq 0$ (possibly $a_s = \alpha_l = \alpha_k$). If $a_1 a_s =0$, then
      \[
      \cdots \to P_{i_1} \xrightarrow{a_1} P_{i_2} \to \cdots  \to P_{i_{s-1}} \xrightarrow{a_{s-1}} P_{i_s} \xrightarrow{a_s} P_{i_1} \xrightarrow{a_1} P_{i_2} \to \cdots
      \]
      is a complex and the algorithm terminates. If $a_1 a_s \neq 0$, then the algorithm terminates with the complex
      \[
      \cdots \to P_{i_s} \xrightarrow{a_1 a_s} P_{i_2} \to \cdots \to P_{i_{s-1}} \xrightarrow{a_{s-1}} P_{i_s} \xrightarrow{a_1 a_s} P_{i_2} \to \cdots. \qedhere
      \]
\end{proof}

\begin{remark}
   Again, it is evident that both Proposition \ref{existenceofperiodiccomplexsufficient} and Theorem \ref{orientedcyclesuffices} hold true for any power of $\Sigma$.
\end{remark}

In light of \cite[Theorem 1]{MR2184464}, and in analogy to \cite[Conjecture 1.2]{MR3024264}, the highest hope one could have for the search for more algebras whose orbit categories do not coincide with their triangulated hulls, is the following.

\begin{conjecture}
   For each positive integer $n$, the embedding
   \[
   \Delta \colon \Db(\modules \Lambda)/\Sigma^n \to \Dn(\modules \Lambda)
   \]
   is dense if and only if $\Lambda$ is piecewise hereditary.
\end{conjecture}

Recall that the \textit{strong global dimension} of $\Lambda$, a notion first proposed by Ringel, is the supremum of the widths of the indecomposable objects in $\Kb(\proj \Lambda)$. By the work of Happel and Zacharia in \cite{MR2413349}, $\Lambda$ is piecewise hereditary if and only if it has finite strong global dimension. One approach to the conjecture might be to look for a more explicit argument showing that arbitrarily wide indecomposable objects exist in $\Kb(\proj \Lambda)$ whenever $\Lambda$ is not piecewise hereditary. Indeed, if one can show that one of these indecomposable complexes is even periodic, then the conjecture will be settled by Proposition \ref{existenceofperiodiccomplexsufficient}.

In a sense, this has already been done in the commutative setting. That is, in \cite{MR3272099} Buchweitz and Flenner classify the commutative noetherian rings of finite strong global dimension. As finiteness of the strong global dimension is a local property, the characterization is furnished by the construction of an `iterated Koszul complex'. Explicitly, if $(R, \mathbf m)$ is local and admits a regular sequence of length $l \geq 2$, then the associated Koszul complex
\[
K = 0 \to K^0 \to K^1 \to \cdots \to K^{l-1} \to K^l \to 0
\]
can be `spliced' with $\Sigma^l K$ by taking the cone of the chain map $K \to \Sigma^l K$ whose only non-zero component is an isomorphism $K^0 \to K^l$. Factoring out an acyclic subcomplex isomorphic to $0 \to K^0 \to K^l \to 0$ yields an indecomposable complex, and it is clear how to repeat the procedure in both directions to obtain periodicity.

Secretly, a modification of the iterated Koszul complex appeared already in the algorithm in the proof of Theorem \ref{orientedcyclesuffices}. Let us explain how this works, as a rewriting of the process in homological terms will indicate that a characterization of piecewise hereditary algebras by minimal projective resolutions of simples may prove helpful in settling our conjecture. Also, the reader might find the rephrased algorithm to be more easily applied in certain examples. So assume that the algebra $\Lambda$ admits a simple module $S$ with $\Ext_{\Lambda}^l(S,S) \neq 0$. Then a chain map $f\colon P \to \Sigma^l P$ which is not null-homotopic must exist, where $P$ is a minimal projective resolution of $S$. Note that this is precisely what happens with the simple $R$-module $R/\mathbf m$ above, whose minimal projective resolution is nothing but the Koszul complex $K$. Factoring out the only acyclic subcomplex of the cone $C_f$ results in some complex $C$, and a decomposition $C = C_1 \oplus C_2$ will induce a decomposition in homology, which is simply $\Sigma S \oplus \Sigma^l S$. Hence we may assume that $H(C_1) = \Sigma S$ and $H(C_2) = \Sigma^l S$, which implies $C_1 = \Sigma P$ and $C_2 = \Sigma^l P$, since $P$ is indecomposable. This is a contradiction, since a splitting of the canonical exact sequence
\[
0 \to \Sigma^l P \to C_f \to \Sigma P \to 0
\]
of complexes implies the vanishing of $f$ in the homotopy category. Hence $C$ is indecomposable, and iterating results in periodicity. More generally, a cycle of non-zero extensions between simples allows a similarly flavored combination of minimal projective resolutions, yielding an indecomposable periodic complex of projectives.

\begin{example}
   Let $G$ be the quiver
   \begin{center}
      \begin{tikzpicture}
         \matrix(a)[matrix of math nodes,
         row sep=2em, column sep=2em,
         text height=1.5ex, text depth=0ex]
         {1 & 2 & 3 & 4. \\};
         \path[->,font=\scriptsize](a-1-1) edge node[above]{$\alpha$}(a-1-2)
         (a-1-2) edge node[above]{$\beta$}(a-1-3)
         (a-1-3) edge node[above]{$\gamma$}(a-1-4);
         \path[->,font=\scriptsize](a-1-4) edge [in=315, out=225, looseness = .22] node[below]{$\delta$} (a-1-1);
      \end{tikzpicture}
   \end{center}
   If $\Lambda = \mathbb k G/(\beta \alpha, \delta \gamma)$, then $\Ext_{\Lambda}^2(S_1, S_3) \neq 0 \neq \Ext_{\Lambda}^2(S_3, S_1)$, corresponding to
   \begin{center}
      \begin{tikzpicture}[baseline=(current bounding box.center)]
         \matrix(a)[matrix of math nodes,
         row sep=1.5em, column sep=1.5em,
         text height=1.5ex, text depth=0.25ex]
         { & &  &  & 0 & P_3 & P_4 & P_1 & 0 \\
          &  & 0 & P_1 & P_2 & P_3 & 0 &  &  \\
         0 & P_3 & P_4 & P_1 & 0. &  &  &  &  \\};
         \path[->,font=\scriptsize](a-1-5) edge (a-1-6)
         (a-1-6) edge node[above]{$\gamma$}(a-1-7)
         (a-1-7) edge node[above]{$\delta$}(a-1-8)
         (a-1-8) edge (a-1-9)

         (a-2-3) edge (a-2-4)
         (a-2-4) edge node[above]{$\alpha$}(a-2-5)
         (a-2-5) edge node [above]{$\beta$}(a-2-6)
         (a-2-6) edge (a-2-7)

         (a-3-1) edge (a-3-2)
         (a-3-2) edge node[above]{$\gamma$}(a-3-3)
         (a-3-3) edge node[above]{$\delta$}(a-3-4)
         (a-3-4) edge (a-3-5)

         (a-1-6) edge node[right]{$1_{P_3}$}(a-2-6)
         (a-2-4) edge node[right]{$1_{P_1}$}(a-3-4);
      \end{tikzpicture}
   \end{center}
   Taking cones and factoring out acyclic subcomplexes yields
   \[
   0 \to P_3 \xrightarrow{\gamma} P_4 \xrightarrow{\alpha \delta} P_2 \xrightarrow{\gamma \beta} P_4 \xrightarrow{\delta} P_1 \to 0,
   \]
   revealing the $2$-periodic
   \[
   \cdots \to P_2 \xrightarrow{\gamma \beta} P_4 \xrightarrow{\alpha \delta} P_2 \xrightarrow{\gamma \beta} P_4 \to \cdots.
   \]
   In like manner, over the algebra $\mathbb k G /(\gamma \beta \alpha)$, there are chain maps
   \begin{center}
      \begin{tikzpicture}[baseline=(current bounding box.center)]
         \matrix(a)[matrix of math nodes,
         row sep=1.5em, column sep=1.5em,
         text height=1.5ex, text depth=0.25ex]
         { &  &  & 0 & P_1 & P_3 & P_4 & 0 \\
          &  & 0 & P_4 & P_1 & 0 &  &  \\
          0 & P_1 & P_3 & P_4 & 0. &  &  &  \\};
         \path[->,font=\scriptsize](a-1-4) edge (a-1-5)
         (a-1-5) edge node[above]{$\beta \alpha$}(a-1-6)
         (a-1-6) edge node[above]{$\gamma$}(a-1-7)
         (a-1-7) edge (a-1-8)

         (a-2-3) edge (a-2-4)
         (a-2-4) edge node[above]{$\delta$}(a-2-5)
         (a-2-5) edge (a-2-6)

         (a-3-1) edge (a-3-2)
         (a-3-2) edge node[above]{$\beta \alpha$}(a-3-3)
         (a-3-3) edge node[above]{$\gamma$}(a-3-4)
         (a-3-4) edge (a-3-5)

         (a-1-5) edge node[right]{$1_{P_1}$}(a-2-5)
         (a-2-4) edge node[right]{$1_{P_4}$}(a-3-4);
      \end{tikzpicture}
   \end{center}
   between minimal projective resolutions of simples. Taking cones and factoring out acyclic subcomplexes yields
   \[
   0 \to P_1 \xrightarrow{\beta \alpha} P_3 \xrightarrow{\beta \alpha \delta \gamma} P_3 \xrightarrow{\gamma} P_4 \to 0,
   \]
   hence the $1$-periodic
   \[
   \cdots \to P_3 \xrightarrow{\beta \alpha \delta \gamma} P_3 \xrightarrow{\beta \alpha \delta \gamma} P_3 \to \cdots.
   \]
\end{example}

\subsection{Functors induced on orbit categories} \label{functorsinducedonorbitcategories}

Given a category $\T$ with an automorphism $\F$, one might expect that $\F$ induces the identity functor on the orbit category $\T\!/\F$. Famously, this does happen in the cluster category $\Db(\modules \Lambda)/\amsbb S \circ \Sigma^{-2}$, and hence the mantra `the cluster category is $2$-Calabi-Yau'. Interestingly enough however, there are examples that do not conform to this intuition. As a matter of fact, our investigations reveal that $\Sigma$ is not necessarily isomorphic to the identity on $\Db(\modules \Lambda)/\Sigma$. Let us elaborate, under the necessary assumption $\characteristic (\mathbb k) \neq 2$.

First of all, note that the condition $\Sigma \cong \id$ on a triangulated category is a highly restrictive one. In fact, as was pointed out to the author by Steffen Oppermann, it essentially implies semisimplicity. Hence we should not expect to find such an isomorphism of functors on $\D1(\modules \Lambda) \cong \K1(\proj \Lambda)$. And indeed, recalling that the shift functor changes the signs of differentials, it is easy to produce an algebra $\Lambda$ and an object in $\K1(\proj \Lambda)$ that does not admit an isomorphism to its shift, never mind one that commutes with induced maps. Of course, one might still suspect that $\Sigma$ coincides with the identity functor on the orbit category $\Db(\modules \Lambda)/\Sigma$ itself, i.e.\ on the subcategory of gradables in $\K1(\proj \Lambda)$. And admittedly, each gradable does allow an isomorphism to its shift (even worse, there are often many). However, objectwise isomorphisms $\Sigma P \cong P$ cannot constitute a natural isomorphism $\Sigma \cong \id$ of functors on $\Db(\modules \Lambda)/\Sigma$ in general. For example, consider the path algebra of
\begin{center}
   \begin{tikzpicture}
      \matrix(a)[matrix of math nodes,
      row sep=2em, column sep=2em,
      text height=1.5ex, text depth=0ex]
      {1 & 2 & 3 \\};
      \path[->,font=\scriptsize](a-1-1) edge node[above]{$\alpha$}(a-1-2)
      (a-1-2) edge node[above]{$\beta$}(a-1-3);
      \path[->,font=\scriptsize](a-1-1) edge [in=225, out=315, looseness = .3] node[below]{$\gamma$} (a-1-3);
   \end{tikzpicture}
\end{center}
modulo to ideal generated by $\beta \alpha$, and the gradable differential modules
   \[
   Q_1 = \bigl(P_1, 0\bigr) \enspace \text{and} \enspace Q_2 = \Biggl(P_2 \oplus P_3, \biggl(\begin{matrix} 0 & 0 \\ \beta & 0 \end{matrix} \biggr) \Biggr).
   \]
If $\Sigma \cong \id$ as functors on the orbit category, then the morphism
   \[
   f \colon Q_1 \xrightarrow{\bigl(\begin{smallmatrix} \alpha \\ \gamma \end{smallmatrix}\bigr)} Q_2
   \]
   fits in a commutative square
   \begin{equation} \label{naturalsquare}
      \begin{tikzpicture}[baseline=(current bounding box.center)]
      \matrix(a)[matrix of math nodes,
      row sep=2.5em, column sep=3em,
      text height=1.5ex, text depth=0.25ex]
      { Q_1 & Q_2 \\
       \Sigma Q_1 & \Sigma Q_2 \\};
       \path[->,font=\scriptsize](a-1-1) edge (a-2-1)
       (a-1-2) edge (a-2-2)
       (a-1-1) edge node[above]{$f$}(a-1-2)
       (a-2-1) edge node[above]{$\Sigma f = f$}(a-2-2);
      \end{tikzpicture}
   \end{equation}
   whose vertical arrows are isomorphisms. As this occurs in a triangulated category, an isomorphism of cones $C_f \to C_{\Sigma f}$ is induced, i.e.\
   \[
   \Biggl(P_1 \oplus P_2 \oplus P_3, \begin{pmatrix}0 & 0 & 0 \\ \alpha & 0 & 0 \\ \gamma & \beta & 0 \end{pmatrix}  \Biggr) \cong \Biggl(P_1 \oplus P_2 \oplus P_3, \begin{pmatrix}0 & 0 & 0 \\ \alpha & 0 & 0 \\ \gamma & -\beta & 0 \end{pmatrix}  \Biggr).
   \]
   However, a straightforward calculation shows that the latter isomorphism does not exist. One can also rebut the existence of (\ref{naturalsquare}) with no mention of the axioms of triangulated categories, using only the facts that an isomorphism $Q_1 \to \Sigma Q_1$ is multiplication by $\lambda e_1$ with $\lambda \in \mathbb k^{\ast}$ and that each isomorphism $Q_2 \to \Sigma Q_2$ is some
   \[
   \begin{pmatrix} \lambda' e_2 & 0 \\ \lambda'' \beta & -\lambda' e_3 \end{pmatrix}
   \]
   with $\lambda' \in \mathbb k^{\ast}$ and $\lambda'' \in \mathbb k$.

\begin{remark}
   Similarly, as long as $n$ is odd, $\Sigma^n$ does not in general coincide with the identity functor on $\Db(\modules \Lambda)/\Sigma^n$. In light of this, one could (very naively) worry that the $n$-Calabi-Yau property of the \textit{$n$-cluster category} $\Db(\modules \Lambda)/\amsbb S \circ \Sigma^{-n}$, introduced in \cite{MR2184464}, might in fact fail unless $n$ is even. In this case, however, Keller shows in  \cite{keller2008corrections} that the signs do add up to ensure $\amsbb S \cong \Sigma^n$.
\end{remark}

This behavior can be traced back to the formal DG algebra $\Gamma = \Lambda[t, t^{-1}]$ with $|t|=1$, employed in Section \ref{sectionthehull}. Recall that if $M$ is a right DG module over a DG algebra $\dgA$, then the right DG $\dgA$-module $\Sigma M$ is the shifted complex with $\dgA$-action
\[
m \ast a = ma.
\]
On the other hand, if $N$ is a left DG $\dgA$-module, then the shifted complex $\Sigma N$ is a left DG $\dgA$-module with action given by
\[
a \ast n = (-1)^{|a|}an.
\]

\begin{lemma} \label{notisoasbimodules}
   $\Gamma$ and $\Sigma \Gamma$ are isomorphic as right and as left DG $\Gamma$-modules, but not as DG $\Gamma$-bimodules. In particular $\Sigma \cong - \otimes_{\Gamma} \Sigma \Gamma$ is not isomorphic to the identity functor on the category of right DG $\Gamma$-modules.
\end{lemma}

\begin{proof}
   Let $f \colon \Gamma \to \Sigma \Gamma$ be a homomorphism of right DG $\Gamma$-modules, i.e.\ $f =(f_i)$ where $f_i \in \End_{\Lambda}(\Lambda)$ for each integer $i$ such that
   \[
   f_{|\gamma \gamma'|}(\gamma \gamma') = f_{|\gamma|}(\gamma) \ast \gamma' = f_{|\gamma|}(\gamma) \gamma'
   \]
   for homogeneous $\gamma, \gamma' \in \Gamma$. In particular picking $\gamma' =1_{\Lambda} \in \Gamma_1$ yields $f_{|\gamma|+1} = f_{|\gamma|}$,
   i.e.\ $f$ must be given by the same homomorphism in each degree. Similarly, if $g = (g_i)\colon \Gamma \to \Sigma \Gamma$ is a homomorphism of left DG $\Gamma$-modules, then
   \[
   g_{|\gamma' \gamma|}(\gamma' \gamma) = \gamma' \ast g_{|\gamma|}(\gamma) = (-1)^{|\gamma'|}\gamma' g_{|\gamma|}(\gamma),
   \]
   and picking $\gamma' = 1_\Lambda \in \Gamma_1$ yields $g_{|\gamma|+1} = - g_{|\gamma|}$. Clearly, $\Gamma \cong \Sigma \Gamma$ on each side, but the different signs required by $f$ and $g$ imply that no isomorphism of DG $\Gamma$-bimodules can exist.
\end{proof}

\bibliographystyle{amsplain}
\bibliography{main}
\end{document}